\documentclass[11pt,a4paper]{article}

\usepackage[asymmetric]{geometry}
\usepackage{titling}

\usepackage{authblk}

\usepackage[margin=1.5cm, font=small,labelfont=bf]{caption}
\usepackage[utf8]{inputenc}
\usepackage{amsmath,amsthm,amsfonts,amssymb}
\usepackage{graphicx}
\usepackage{color}
\usepackage{url}
\usepackage{booktabs}
\usepackage{enumerate}
\usepackage{cases}
\usepackage{multirow}
\usepackage{braket}
\usepackage{bbding}
\usepackage{cancel}  

\usepackage{hyperref}
\hypersetup{ colorlinks = true, citecolor = {blue}, linkcolor = {magenta}}
 
\usepackage{cleveref}

\usepackage{pgfplots}
\pgfplotsset{compat=newest}

\usepackage{tikz}
\usetikzlibrary{arrows}
\usetikzlibrary{calc}
\usetikzlibrary{intersections}
\usetikzlibrary{shapes,snakes}

\newtheorem{lemma}{Lemma}[section]
\newtheorem{theorem}{Theorem}[section]

\newtheorem{definition}{Definition}[section]
\newtheorem{example}{Example}[section]

\numberwithin{equation}{section}

\usepackage{algorithm, algorithmic}

\def\C{\mathbb C}
\def\R{\mathbb R}
\def\Re{\mathrm{Re}}

\def\lo{\mathbf{o}}
\def\res{\mathrm{res}}

\def\spr{\rho} 
\def\rqs{g} 
\def\objf{F} 

\DeclareMathOperator*{\argmax}{arg\,max}

\DeclareSymbolFont{yhlargesymbols}{OMX}{yhex}{m}{n}
\DeclareMathAccent{\wideparen}{\mathord}{yhlargesymbols}{"F3}

\title{
Variational Characterization of Monotone Nonlinear Eigenvector Problems 
and Geometry of Self-Consistent-Field Iteration 
}

\author{
	Zhaojun Bai\thanks{
	Department of Computer Science, University of
	California at Davis, Davis, CA 
	(\url{zbai@ucdavis.edu}).
}
	~and Ding Lu\thanks{
Department of Mathematics, University of Kentucky,
	Lexington, KY 
	(\url{Ding.Lu@uky.edu}).
}
}
\date{version date September 20, 2022}

\graphicspath{{./FIG/}}

\begin{document}
\maketitle

\begin{abstract}
This paper concerns a class of
monotone eigenvalue problems with eigenvector nonlinearities (mNEPv).
The mNEPv is encountered in applications such as 
the computation of joint numerical radius of matrices,
best rank-one approximation of third order partial symmetric tensors, 
and distance to singularity for dissipative Hamiltonian 
differential-algebraic equations. 
We first present a variational characterization of the mNEPv.
Based on the variational characterization, we provide 
a geometric interpretation of the self-consistent-field (SCF)
iterations for solving the mNEPv, prove the global
convergence of the SCF, and devise an accelerated SCF. 
Numerical examples from a variety of applications    
demonstrate the theoretical properties
and computational efficiency of the SCF and its acceleration. 
\end{abstract}


\medskip


\tableofcontents

\section{Introduction}

We consider the following eigenvector-dependent nonlinear eigenvalue problem: 
\begin{equation}\label{eq:nepv}
H(x)\, x = \lambda\,x,
\end{equation}
where $H(x)$ is a Hermitian matrix-valued function of the form
\begin{equation}\label{eq:hx}
H(x):= \sum_{i=1}^m h_i(x^HA_ix) \, A_i,
\end{equation}
and $A_1,\dots,A_m$ are $n$-by-$n$ Hermitian matrices, 
$h_1,\dots,h_m $ are differentiable and non-decreasing functions over $\R$.
The goal is to find a unit-length vector $x\in\C^n$ 
and a scalar $\lambda\in\R$ satisfying \eqref{eq:nepv} 
and, furthermore, $\lambda\, (= x^H H(x) x)$ is the largest eigenvalue of $H(x)$.
The solution vector $x$ is called an eigenvector of the eigenvalue problem~\eqref{eq:nepv}
and $\lambda$ is the corresponding eigenvalue.
Since $H(\gamma x) \equiv H(x)$ for any $\gamma\in\C$ 
with $|\gamma|=1$, if $x$ is an eigenvector, then so is $\gamma x$. 

The matrix-valued function $H(x)$ in~\eqref{eq:hx} is a linear combination
of constant matrices $A_1,\ldots, A_m$
with monotonic functions $h_1, \ldots, h_m$. 
We say $H(x)$ is of a monotone affine-linear structure
and, for simplicity, call the eigenvalue problem~\eqref{eq:nepv} 
a monotone NEPv, or mNEPv.

In \Cref{sec:varnepv}, we will see that the mNEPv~\eqref{eq:nepv} 
is intrinsically related to the following maximization problem:
\begin{equation} \label{eq:maxf}
\max_{x \in \C^n,\, \|x\|=1}
\Big\{ \objf(x) := \sum_{i=1}^m \phi_i\left(x^HA_ix\right) \Big\},  
\end{equation}
where $\phi_i$ are the anti-derivatives of $h_i$, i.e., 
$\phi^{'}_i(t)  = h_i(t)$, for $i=1,\dots,m$. 
Since $h_i$ are differentiable and non-decreasing, 
$\phi_i$ are twice-differentiable and convex functions. 
We call~\eqref{eq:maxf} an associated maximization
of the mNEPv~\eqref{eq:nepv}, or aMax.

\medskip 
The mNEPv~\eqref{eq:nepv} is a special class of 
the eigenvalue problems with eigenvector nonlinearities (NEPv).  
NEPv have been extensively studied in the Kohn--Sham 
density functional theory 
for electronic structure calculations~\cite{Martin:2004,Szabo:2012}
and the Gross--Pitaevskii eigenvalue problem, which is a
nonlinear Schr\"{o}dinger equation used in quantum physics to 
describe the ground states of 
ultracold bosonic gases~\cite{Bao:2004,Jarlebring:2014}.
NEPv have also been found in a variety of computational problems 
in data science. For examples, 
Fisher's linear discriminant analysis~\cite{Ngo:2012,Zhang:2013,Zhang:2014}
and its robust version~\cite{Bai:2018}, 
spectral clustering using the eigenpairs of the $p$-Laplacian~\cite{Hein:2009}, 
core-periphery detection in networks~\cite{Tudisco:2019}, 
and orthogonal canonical correlation analysis~\cite{Zhang:2022}.  

Self-Consistent-Field (SCF) iteration is a gateway algorithm
to solve NEPv, much like the power method for solving 
linear eigenvalue problems.  The SCF was introduced in computational
physics back to 1950s~\cite{Roothaan:1951}.
Since then, the convergence analysis of the SCF has long been 
an active research topic in the study of NEPv;
see~\cite{Bai:2022,Cai:2018,Cances:2021,Cances:2000,Liu:2014,Liu:2015,Stanton:1981,Upadhyaya:2021,Yang:2009}.

Although the underlying structure of the mNEPv~\eqref{eq:nepv}
is commonly found in NEPv, it has been largely unexploited in 
previous studies. In this paper, we will conduct an in-depth 
systematical study of the mNEPv and exploit its underlying 
structure.  Theoretically,  we will develop a variational characterization 
of the mNEPv~\eqref{eq:nepv} by maximizers of the aMax~\eqref{eq:maxf}.
Using the variational characterization, we will provide 
a geometric interpretation of the SCF for solving the mNEPv~\eqref{eq:nepv}.
The geometry of the SCF reveals the global convergence of the algorithm. 
Consequently, we will prove the globally monotonic convergence of the SCF. 
Combining with the local convergence analysis of the SCF from previous work, 
we have a full understanding of the types of eigenvectors that the SCF computes.
Finally, we will present an accelerated SCF iteration by exploiting 
the underlying structure of $H(x)$, and demonstrate its 
efficiency with examples from a variety of applications.

The aMax~\eqref{eq:maxf} is interesting in its own right and finds
numerous applications.  One important source of the problems is a
quartic maximization over the Euclidean ball, where $\phi_i(t) = t^2$~\cite{Nesterov:2003}.
In~\Cref{sec:nepvegs}, we will discuss
such quartic maximization problems arising from the joint numerical radius 
computation and the rank-one approximation of partial-symmetric 
tensors. Another application of the aMax~\eqref{eq:maxf} is 
an optimization from the study of distance to singularity for 
dissipative Hamiltonian differential-algebraic equation (dHDAE) systems
and the related higher order dynamical systems~\cite{Mehl:2021}.
In addition, the aMax~\eqref{eq:maxf} arises in robust optimization 
with ellipsoid uncertainty; see e.g.,~\cite{Ben:1998}.
By the intrinsic connection between the mNEPv and the aMax,
we devise an eigenvalue-based approach for solving the aMax 
that can exploit state-of-the-art eigensolvers from numerical linear algebra.

We note that optimization problems involving a function $\objf(x)$ in the form 
of~\eqref{eq:maxf} have been investigated in the literature. 
But such problems are often formulated as the minimization of $\objf(x)$ 
over the real vector space $\mathbb{R}^n$, instead of the maximization 
of $\objf(x)$ over the complex vector space $\mathbb{C}^n$.
Examples of recent studies include the quartic-quadratic 
optimization with $\phi_i(t) = t^2$ or $t$~\cite{Huang:2022,Zhang:2021}
and the Crawford number computation with $\phi_i(t) = t^2$~\cite{Lu:2020}.
For these minimization problems, the eigenvalue-based approaches have
been developed, and the resulting NEPv are governed by 
the equation~\eqref{eq:nepv} with $H(x)$ of the form~\eqref{eq:hx};
see, e.g.,~\cite{Huang:2022,Lu:2020}.  The major distinction is that the target eigenvalue
$\lambda$ is corresponding to the smallest eigenvalue of $H(x)$.
Consequently, the solution and analysis of the resulting NEPv are 
fundamentally different form the mNEPv~\eqref{eq:nepv}. For example, 
the SCF is no longer globally convergent for computing the smallest eigenvalue.

The rest of this paper is organized as follows. 
In~\Cref{sec:varnepv}, we present a variational characterization 
of the mNEPv~\eqref{eq:nepv} by maximizers of the aMax~\eqref{eq:maxf}. 
\Cref{sec:scf} is devoted to the SCF,
where we introduce a geometric interpretation of the SCF,
prove its global convergence,
and discuss an acceleration scheme with implementation issues.
\Cref{sec:nepvegs} is on the applications of the mNEPv~\eqref{eq:nepv}.
Numerical experiments are presented in~\Cref{sec:numex}
and concluding remarks are in~\Cref{sec:conclusion}.

\medskip 
We follow standard notations in matrix computation.
$\R^{m\times n}$ and $\C^{m\times n}$ are the sets of $m$-by-$n$ real
and complex matrices, respectively. 
$\Re(\cdot)$ extracts the real part of a complex matrix or a number.
For a matrix (or a vector) $X$,
$X^T$ stands for transpose,
$X^H$ for conjugate transpose,
$X(i,j)$ for the $(i,j)$-th entry of $X$,
and $\|X\|$ for the matrix 2-norm.
We use $\lambda_{\min}(X)$ and $\lambda_{\max}(X)$
for the smallest and largest eigenvalues of a Hermitian $X$.
The spectral radius (i.e., largest absolute value of eigenvalues)
of a matrix or linear operator is denoted by $\spr(\cdot)$.
Standard little o and big O notations are used: 
$f(x) = \lo(g(x))$  is interpreted as $f(x)/g(x) \to 0$ as $x\to 0$
and 
$f(x) = \mathcal O(g(x))$  is interpreted as $f(x)/g(x) \leq c $ for
some constant $c$ as $x\to 0$.
Other notations will be explained as used.


\section{Variational characterization}\label{sec:varnepv}
Variational characterizations provide powerful tools to the study of
eigenvalue problems, facilitating both theoretical analysis 
and numerical computations. 
A prominent example is the Hermitian linear eigenvalue problem of
the form $Ax = \lambda x$. The variational characterizations, 
also known as the Courant-Fischer principle, of the eigenvalues 
of $A$ are formed using optimizers of the Rayleigh quotient $x^HAx/x^Hx$;
see, e.g.,~\cite{Bhatia:2013,Parlett:1998, Stewart:1990}.
Consequently, bounds for eigenvalues, interlacing and monotonicity 
of eigenvalues can be proved easily with the characterizations.
Variational characterizations have also been developed for 
eigenvalue-dependent nonlinear eigenvalue problems 
of the form $T(\lambda) x = 0$; see a recent survey~\cite{Lampe:2022}
and references therein.
It is also well-known that the NEPv in Kohn-Sham density functional theory 
is derived from the minimization of an energy function in 
electronic structure calculations; see, e.g.,~\cite{Martin:2004,Cances:2021}.
In this section, we provide a variational characterization of
the mNEPv~\eqref{eq:nepv} by exploring its relation to 
the aMax~\eqref{eq:maxf}. 

\subsection{Stability of eigenvectors}\label{sec:stability}

We start with the following NEPv, 
without assuming the structure of $H(x)$ and the 
order of the eigenvalue $\lambda$:
\begin{equation} \label{eq:nepv_general}
H(x) x = \lambda x \quad\text{with}\quad \|x\|=1,
\end{equation} 
where $H(x)$ is Hermitian, differentiable (w.r.t. both real and
imaginary parts of $x$), and unitarily scaling invariant
(i.e., $H(\gamma x) = H(x)$ for any $\gamma \in \C$ with $|\gamma| =1$).
An eigenvector $x$ of the NEPv~\eqref{eq:nepv_general} can be viewed as
an equivalent class 
$[x] :=\{\,\gamma x \mid  \gamma\in\C,\, |\gamma|=1\,\}$,
i.e., a point in the Grassmannian  $\mbox{Gr}(1, \C^n)$,
also known as the complex projective space $\mathbb C\mathbb P^{n-1}$.

Let $x_*$ be an eigenvector of the NEPv~\eqref{eq:nepv_general},
and the corresponding $\lambda_*$ be the $p$-th largest eigenvalue
of $H(x_*)$. Assume $\lambda_*$ is a simple eigenvalue.
Then $[x_*]$ can be interpreted as a solution to the
fixed-point equation over $\mbox{Gr}(1, \C^n)$:
\begin{equation}\label{eq:fixpt}
[x] = \Pi([x]),
\end{equation}
where the mapping $\Pi: \mbox{Gr}(1, \C^n)\to \mbox{Gr}(1, \C^n)$ is
defined by
\begin{equation}\label{eq:pi}
\Pi([x]) := [u(x)], 
\end{equation}
where $u(x)$ is an (arbitrary) unit eigenvector for the $p$-th 
largest eigenvalue of $H(x)$.

To consider the contractivity of the mapping $\Pi$~\eqref{eq:pi} at $[x_*]$, 
we first denote the eigenvalue decomposition of $H(x_*)$ as
\begin{equation}\label{eq:hxeig}
H(x_*) \begin{bmatrix}x_* & X_{*\bot}\end{bmatrix} = 
    \begin{bmatrix}x_* & X_{*\bot}\end{bmatrix} 
    \begin{bmatrix}\lambda_* & \\ & \Lambda_{*\bot} \end{bmatrix},
\end{equation}
where $\begin{bmatrix}x_* & X_{*\bot}\end{bmatrix}\in\C^{n\times n}$
is unitary and $\Lambda_{*\bot}\in \R^{(n-1)\times (n-1)}$ is a diagonal matrix.
We then define an $\R$-linear operator
\footnote{An operator $\mathcal L\colon \C^{m} \to \C^{m}$ 
is called $\R$-linear if 
$\mathcal L(\alpha x+\beta y) = \alpha \mathcal L(x) + \beta\mathcal L (y)$
for all $\alpha,\beta\in\R$ and $x,y\in\C^{m}$.
} 
\begin{equation}\label{eq:lop1}
    \mbox{ $\mathcal L: \C^{n-1} \to \C^{n-1}$ \quad with\quad }
    \mathcal L(z)
    = 
    D_*^{-1} X_{*\bot}^H\left(\, {\bf D}\!H(x_*)[\,X_{*\bot}z\,]\,\right)x_*,
\end{equation}
where $D_*=\lambda_* I_{n-1} - \Lambda_{*\bot}$ is diagonal and
non-singular since $\lambda_*$ is a simple eigenvalue,
and ${\bf D}\!H(x)[\,d\,]$ is the derivative 
of $H$ at $x$ along the direction of $d$:  
\begin{equation}\label{eq:dh1}
{\bf D}\!H(x)[\,d\,] 
:= \lim_{\alpha \in\R,\ \alpha\to 0} \frac{H(x+\alpha d)-H(x)}{\alpha}.
\end{equation}
Let $\rho(\mathcal L)$ be 
the spectral radius of $\mathcal L$, i.e.,
the largest absolute value of the eigenvalues of $\mathcal{L}$.
Then by~\cite[Thm. 4.2]{Bai:2022}, we know that
if $\spr(\mathcal L)<1$, then the fixed-point mapping
$\Pi$~\eqref{eq:pi} is locally contractive at $x_*$;
If $\spr(\mathcal L)>1$, then $\Pi$ is non-contractive at $x_*$;
If $\spr(\mathcal L)=1$, then no immediate conclusion can be drawn for the
contractivity of $\Pi$.\footnote{ 
We mention that the theorem in~\cite[Thm. 4.2]{Bai:2022} is 
stated for the special case of $\lambda_*=\lambda_n$ being the smallest
eigenvalue of $H(x_*)$, but the result holds for a general $p$-th
eigenvalue.}


Return to the mNEPv~\eqref{eq:nepv}, 
in the following lemma, we can show that  
the operator $\mathcal L$ in~\eqref{eq:lop1} is
self-adjoint and positive semi-definite. 
Consequently, it allows us to describe
the contractivity conditions, namely  
$\rho(\mathcal L) < 1$ or $\rho(\mathcal L) \leq 1$,
through the definiteness of a characteristic function. 
We first denote by $\C^{n-1}(\R)$ the vector space $\C^{n-1}$ 
over the field of real numbers $\R$ and by 
\begin{equation}\label{eq:inner}
\langle\, y,z\,\rangle_{\!D}:= \Re(\,y^HDz\,)
\end{equation}
an 
inner product over $\C^{n-1}(\R)$ for a given Hermitian positive definite
matrix $D$ of size $n-1$.

\begin{lemma}\label{lem:regular}
Let $x_*\in \C^n$ be an eigenvector of the mNEPv~\eqref{eq:nepv} 
with a simple eigenvalue $\lambda_*$. 
Then the $\R$-linear operator $\mathcal L$ in~\eqref{eq:lop1} is
self-adjoint and positive semi-definite 
over $\C^{n-1}(\R)$ in the inner product~\eqref{eq:inner}
with $D_*=\lambda_* I_{n-1}-\Lambda_{*\bot}$.  
Moreover, 
\begin{enumerate}[(a)]
    \item\label{lem:regular:a}
	$\spr(\mathcal L) < 1$
    if and only if
    $\varphi(d\,;x_*) < 0$ for all $d\neq 0$ and $d^Hx_* = 0$;

	\item\label{lem:regular:b}
	$\spr(\mathcal L) \leq 1$ if and only if 
	$\varphi(d\,;x_*) \leq 0$ for all $d\neq 0$ and $d^Hx_* = 0$.
	\end{enumerate}
	Here, $\varphi(d\,;x_*)$ is a quadratic function in $d\in \C^n$
	and is parameterized by $x_*$:
	\begin{equation}\label{eq:phifun}
	\varphi(d\,;x_*):= d^H \Big( H(x_*) - (x^H_* H(x_*) x_*)\, I \Big)d +
	2\sum_{i=1}^m h'_i(x^H_*A_ix_*) \cdot (\Re (d^H A_ix_*))^2.
	\end{equation}
\end{lemma}
\begin{proof}
To show that $\mathcal L$ is self-adjoint and positive semi-definite,
we first derive from the definition~\eqref{eq:hx} of $H(x)$
that the directional derivative~\eqref{eq:dh1} is given by
\begin{equation*}\label{eq:dh2}
{\bf D}\!H(x)[\,d\,] = 2\sum_{i=1}^m
\Re\left(\,x^HA_id\,\right) \cdot h'_i(x^HA_ix)\cdot A_i.
\end{equation*} 
Therefore the $\mathbb R$-linear operator $\mathcal{L}$ in~\eqref{eq:lop1} 
takes the form of 
\begin{align}\label{eq:lop2}
	\mathcal L(z) 
	&= 2D_*^{-1} \sum_{i=1}^m  \Re(x_*^HA_iX_{*\bot}z) \cdot
	h'_i(x_*^HA_ix_*)\cdot X_{*\bot}^H A_ix_*.
\end{align}
Since $\lambda_*$ is a simple largest eigenvalue,
$D_*=\lambda_* I_{n-1}-\Lambda_{*\bot}$ is a diagonal and positive
definite matrix.
A quick verification shows
\begin{equation} \label{eq:adj}
\text{ $\langle\, \mathcal L(y),z\,\rangle_{\!D_*} 
= 2\sum_{i=1}^m   h'_i(x_*^HA_ix_*) \cdot \Re(x_*^HA_iX_{*\bot}z)
\cdot \Re(x_*^HA_iX_{*\bot}y)
=  \langle\, y,\mathcal
L(z)\,\rangle_{\!D_*}$,}
\end{equation} 
i.e., 
$\mathcal L: \C^{n-1}(\R)\to \C^{n-1}(\R)$ is self-adjoint
w.r.t. the inner product $\langle \cdot,\,\cdot\rangle_{\!D_*}$.

Let $y=z$ in~\eqref{eq:adj}. We obtain 
\begin{equation} \label{eq:posdef} 
\langle\, z,\mathcal L(z)\,\rangle_{\!D_*} 
= 2\sum_{i=1}^m   h'_i(x_*^HA_ix_*) \cdot \Re(x_*^HA_iX_{*\bot}z)^2 \geq 0,
\end{equation} 
where we used the assumption that $h_i$ is non-decreasing.
By~\eqref{eq:adj} and~\eqref{eq:posdef},
$\mathcal L$ is a self-adjoint and positive semi-definite operator.

Now by the eigenvalue variational principle for self-adjoint operators
(see, e.g.,~\cite[Chap~1]{Weinstein:1972}), the spectral radius
\begin{equation}\label{eq:rphmax}
    \spr(\mathcal L) = 
    \lambda_{\max}(\mathcal L) = 
    \max_{z\neq 0} 
    \frac{\langle\, \mathcal L(z), z\,\rangle_{\!D_*}}
	{\langle\, z,z\,\rangle_{\!D_*}}.
\end{equation}
Let $d=X_{*\bot}z$. Then we have
\begin{equation}\label{eq:lzz}
\langle\, z,z\,\rangle_{\!D_*} \equiv z^H(\lambda_* I_{n-1}-\Lambda_{*\bot})z =
d^H(x_*^H H(x_*) x_*\cdot I_n - H(x_*))d,
\end{equation}
where we used the identities $\lambda_* = x_*^H H(x_*) x_*$ and $H(x_*)X_{*\bot}=X_{*\bot}\Lambda_{*\bot}$.
Therefore,
\begin{equation*}
    \spr(\mathcal L) - 1= 
    \max_{z\neq 0} 
    \frac{\langle\, \mathcal L(z), z\,\rangle_{\!D_*}-
		\langle\, z, z\,\rangle_{\!D_*}}
		{\langle \,z, z\,\rangle_{\!D_*}}
	\equiv 
    \max_{z\neq 0,\ d=X_{*\bot}z} 
    \frac{\varphi(d\,;x_*)}{\langle\, z, z\,\rangle_{\!D_*}},
\end{equation*}
where $\varphi$ is from~\eqref{eq:phifun},
and we used~\eqref{eq:posdef} for 
$\langle\, \mathcal L(z), z\,\rangle_{\!D_*}$ 
and~\eqref{eq:lzz} for $\langle\, z,z\,\rangle_{\!D_*}$.
Consequently, $\spr(\mathcal L) < 1$ (or $\spr(\mathcal L) \leq 1$)
if and only if 
$\varphi(d\,; x_*) < 0$ (or $\varphi(d\,; x_*)\leq 0$)
for all $d = X_{*\bot}z$ with $z\neq 0$.
Since $[X_{*\bot},x_*]$ is unitary,
a vector $d = X_{*\bot}z$ for some $z\neq 0$ if and only if 
$d^Hx_* = 0$ with $d\neq 0$.
Results in items~\eqref{lem:regular:a} and~\eqref{lem:regular:b} follow.
\end{proof}

By the standard notion of stability of fixed points
of a mapping in fixed-point analysis, see, e.g.,~\cite{Amann:1976,Berinde:2007}, 
we can classify the stability of the eigenvectors of the mNEPv~\eqref{eq:nepv} 
using the spectral radius $\spr(\mathcal L)$ and, alternatively, 
the characterization function $\varphi$ in \Cref{lem:regular}. 

\begin{definition}\label{def:regular2}
Let $x_*\in \C^n$ be an eigenvector of the mNEPv~\eqref{eq:nepv}
and $\varphi$ be as defined in~\eqref{eq:phifun}.
Then $x_*$ is a \emph{stable eigenvector} 
if $\varphi(d\,;x_*)<0$ for all $d\neq 0$ and $d^Hx_* = 0$, 
and a \emph{weakly stable eigenvector} if
$\varphi(d\,;x_*)\leq 0$ for all $d\neq 0$ and $d^Hx_* = 0$.
Otherwise, $x_*$ is called a \emph{non-stable eigenvector}. 
\end{definition}

Note that \Cref{def:regular2} does not explicitly require 
$\lambda_*(H(x_*))$ is a simple eigenvalue, since the characteristic
function $\varphi$~\eqref{eq:phifun} is still well-defined in the
case of non-simple eigenvalue.
In addition, we note that for a \emph{stable eigenvector} $x_*$,
the corresponding $\lambda_*$ is necessarily a simple eigenvalue of $H(x_*)$.
Otherwise, there is another eigenvector $\widetilde x$
of $\lambda_*=\lambda_{\max}(H(x_*))$ orthogonal to $x_*$.
By letting $d=\widetilde x$ and recalling $h'_i(t)\geq 0$,
we derive from~\eqref{eq:phifun} that $\varphi(d\,; x_*)\geq 0$, 
which contradicts the condition for a stable eigenvector
that $\varphi(d\,;x_*)<0$ for all $d\neq 0$ and $d^Hx_* = 0$. 

\subsection{Variational characterization}\label{sec:variation}

The following theorem provides a variational characterization of 
the mNEPv~\eqref{eq:nepv} through the aMax~\eqref{eq:maxf}. 
We first recall a standard notion in optimization
(see, e.g.,~\cite[Sec. 2.1]{Nocedal:2006})
that a unit vector $x$ is 
called a \emph{local maximizer} of the aMax~\eqref{eq:maxf}
if there exists $\varepsilon >0$ s.t.
\begin{equation}\label{eq:strictmax}
\objf(x) \geq \objf\left(\frac{x+d}{\|x+d\|}\right)\quad\text{for all $d\in\C^n$ 
with $d^Hx=0$ and $\|d\|\leq \varepsilon $}, 
\end{equation}
and $x$ is a \emph{strict local maximizer} if 
the inequality for $\objf$ in~\eqref{eq:strictmax}
holds strictly. 

\begin{theorem} \label{thm:vc}
Let $x\in\mathbb C^n$ be a unit vector.
\begin{enumerate}[(a)]
\item \label{i:thm:vc:1}
    If $x$ is a stable eigenvector of the mNEPv~\eqref{eq:nepv},
        then $x$ is a strict local maximizer of the aMax~\eqref{eq:maxf}.

\item \label{i:thm:vc:2}
        If $x$ is a local maximizer of the aMax~\eqref{eq:maxf},
        then $x$ is a weakly stable eigenvector of the mNEPv~\eqref{eq:nepv}.
\end{enumerate}
\end{theorem}
\begin{proof} 
Let $\widehat x = (x+d)/\|x+d\|$, then
\[ 
\widehat x^H A_i  \widehat x
= x^H A_i x + \delta_i
\quad\text{for $i=1,2,\dots, m$},
\]
where
\begin{equation}\label{eq:rqexpand}
	\delta_i := 2\cdot\Re (d^HA_ix) 
    + d^H\left(A_i-(x^HA_ix) I\right)d
    +\mathcal O(\|d\|^3).
\end{equation}
Hence the $i$-th term of $\objf\left(\widehat x\right)$ satisfies
\begin{align*}
\phi_i\left(\widehat x^HA_i \widehat x\right) 
= \phi_i(\rqs_i(x) +\delta_i) 
= 
\phi_i(\rqs_i(x) ) + h_i(\rqs_i(x) )\cdot
\delta_i  + \frac{1}{2}{h'_i(\rqs_i(x) )}\cdot \delta_i ^2 + \lo(\delta_i^2), 
\end{align*}
where $\rqs_i(x) := x^H A_i x $.
Summing over all $\phi_i$ from $i=1$ to $m$, we obtain
\begin{align}
    \objf\left(\widehat x\right)
    & \equiv \sum_{i=1}^m 
    \left[
    \phi_i(\rqs_i(x)) + h_i(\rqs_i(x) )\cdot
    \delta_i  + \frac{1}{2}{h'_i(\rqs_i(x))}\cdot\delta_i ^2 + \lo(\delta_i^2)
    \right] \notag \\
    & = \objf(x) + 2\,\Re(d^HH(x)x) + d^H\left(H(x)-s(x)I\right)d
    +  2 \sum_{i=1}^m h'_i(\rqs_i(x))\cdot \left(\Re(d^HA_ix)\right)^2 +
    \lo(\|d\|^2) \notag  \\
& = \objf(x) + 2\, \Re(d^HH(x)x) + \varphi(d\,;x)  + \lo(\,\|d\|^2), \label{eq:ftx} 
\end{align}
where the second equality is by \eqref{eq:rqexpand} and $s(x):= x^HH(x)x$.

For the result~\eqref{i:thm:vc:1}: We need to show 
the inequality~\eqref{eq:strictmax} holds strictly.  
First, it follows from the NEPv $H(x)x=\lambda x$ and the orthogonality $d^Hx =0$
that $d^HH(x)x =0$. 
So the second term on the right side of~\eqref{eq:ftx}
vanishes, and we have
\begin{equation}\label{eq:ftx1}
\objf\left(\widehat x \right) = \objf(x) + \varphi(d\,; x) +
\lo(\|d\|^2).
\end{equation}
Since the stability of $x$ (\Cref{def:regular2}) implies 
$\varphi(d\,;x) <0$ 
and we can drop $\lo(\|d\|^2)$ (which
is negligible to the quadratic $\varphi(d\,;x) = \mathcal O(\|d\|^2)$), 
\eqref{eq:ftx1} leads to 
$\objf(x)>\objf\left(\widehat x \right)$ as $\|d\|\to 0$. 

For the result~\eqref{i:thm:vc:2}: 
Let $d$ be sufficiently tiny and $d^Hx =0$.
It follows from the local maximality~\eqref{eq:strictmax} 
and the expansion~\eqref{eq:ftx} that 
\begin{equation}\label{eq:fdiff}
	0 \geq \objf\left(\widehat x \right) - \objf(x) 
	= 
	2\cdot \Re(d^HH(x)x) + \varphi(d\,; x) +\lo(\|d\|^2).
\end{equation}
Therefore, the leading first-order term must vanish,
that is, $\Re(d^HH(x)x) = 0 $ for all $d$ with $d^Hx =0$.
This implies that $H(x)x$ and $x$ have common null spaces, i.e.,
\begin{equation}\label{eq:stationary}
    H(x)x = \lambda x,
\end{equation}
for some scalar $\lambda$.

To show that $x$ is a weakly stable eigenvector (\Cref{def:regular2}),
we still need to prove that (i) $\lambda$ in~\eqref{eq:stationary} is 
the largest eigenvalue of $H(x)$, and (ii) it holds that 
$ \varphi(d\,; x) \leq 0 $ for all $d$ with $d^Hx=0$.
For condition (ii), we recall that 
the first-order term in~\eqref{eq:fdiff} vanishes, and hence
\[
 \varphi(d\,; x) +\lo(\|d\|^2) \leq 0
 \quad\Rightarrow\quad
\varphi(d\,; x) \leq 0,
\]
where the last equation is due to $\varphi(d\,; x)$ is a quadratic
function in $d$.
Now as condition~(ii) holds, $\lambda$ must be the largest
eigenvalue of $H(x)$.
Otherwise, there is a $\widetilde\lambda > \lambda$ with
$H(x)\widetilde x =\widetilde \lambda \widetilde x$ and $\widetilde x^Hx
= 0$.
Recall~\eqref{eq:phifun} that $\varphi(d\,; x) \geq d^H(H(x) - (x^HH(x)x) I)d$.
Let $d=\widetilde x$ and we have $\varphi(d\,; x) \geq \widetilde \lambda -\lambda > 0$,
contracting $\varphi(d\,; x) \leq 0$.
\end{proof} 

Results from~\Cref{thm:vc} can be regarded as
second-order
sufficient and necessary conditions for the aMax~\eqref{eq:maxf}.
They are stated in a way to highlight the connections between the local
maximizers of the aMax and the stable eigenvectors of the mNEPv, 
which benefits the analysis of the SCF to be discussed in~\Cref{sec:scf}.
We note that the objective function $\objf(x)$ of the aMax~\eqref{eq:maxf} 
is not holomorphic (i.e. complex differentiable in $x\in\C^n$).
Therefore, the second-order KKT conditions 
(see, e.g.,~\cite[Sec.~12.5]{Nocedal:2006}) 
are not immediately 
applicable.\footnote{Turning the problem to a real variable
optimization, in the real and imaginary parts of $x\in\C^n$,
does not fully address the issue, 
since there will be no strict local maximizers in the standard sense 
for a unitarily invariant $\objf(x)$.}


To end this section, let us discuss three immediate implications 
of the variational characterization in~\Cref{thm:vc}.  
First, given the intrinsic connection between 
the mNEPv~\eqref{eq:nepv} and the aMax~\eqref{eq:maxf},
stable and weakly stable
eigenvectors of the NEPv are of particular interest. 
Since the aMax~\eqref{eq:maxf} always has a global (hence local) maximizer,
\Cref{thm:vc}\eqref{i:thm:vc:2} guarantees the existence of weakly 
stable eigenvectors of the mNEPv~\eqref{eq:nepv}.
Although such eigenvectors may not be unique and may correspond to local but
non-global maximizers of the aMax~\eqref{eq:maxf} (see~\Cref{eg:rand}), 
the connection to the aMax~\eqref{eq:maxf} greatly facilitates the
development and analysis of numerical algorithms for the mNEPv~\eqref{eq:nepv},
e.g., the geometric interpretation of the SCF to be discussed in~\Cref{sec:scf}.

Second, \Cref{thm:vc} is a generalization of the well-known variational 
charaterization of Hermitian eigenvalue problem.
Consider the special case of the mNEPv~\eqref{eq:nepv} with 
$m=1$ and $h_1(t) = 1$, we have
\[
\text{``mNEPv'':}\quad A_1x = \lambda x
\quad\text{and}\quad
\text{``aMax'':}\quad \max_{\|x\|=1} x^HA_1x,
\]
where $\lambda$ is the largest eigenvalue of $A_1$.
Let $\lambda\geq\lambda_2\geq\dots \geq \lambda_n$
be the eigenvalues of $A_1$ with the corresponding orthonormal
eigenvectors $[x,x_2,\dots,x_n]$.
Since a non-zero $d$ in the complement of $[x]$ can be 
written as a linear combination $d = \alpha_2x_2 +\dots+\alpha_nx_n$ 
for some coefficients $\{\alpha_i\}_{i=2}^n$, 
the quadratic function $\varphi$ defined in~\eqref{eq:phifun}
becomes 
\[
\varphi(d\,; x) = d^H(A_1 - \lambda I)d 
= \sum_{i=2}^n \alpha_i^2\,(\lambda_i -\lambda). 
\]
Note that $\varphi(d\,; x)$ is non-positive, and strictly negative 
if $\lambda$ is simple. 
Consequently, \Cref{thm:vc} can be paraphrased as follows:
\begin{enumerate}[(a)]
\setlength\itemsep{0em}
\item
If the largest eigenvalue of $A_1$ is simple,
then the corresponding eigenvector $x$ is a strict local maximizer of
the Rayleigh quotient $ (x^HA_1x)/(x^Hx)$;

\item
If $x$ is a local maximizer of the Rayleigh quotient $(x^HA_1x)/(x^Hx)$,
then $x$ is an eigenvector corresponding to the largest eigenvalue
of $A_1$.
\end{enumerate}
Both statements follow from the well-known 
variational characterization of Hermitian eigenvalue problems:
Eigenvectors of the largest eigenvalue of $A_1$ are corresponding to 
local and global maximizers of $(x^HA_1x)/(x^Hx)$; 
If the largest eigenvalue is simple, then its eigenvector (up to scaling) 
is the only maximizer; see, e.g.,~\cite[Sec.4.6.2]{Absil:2009}. 

Third, if the coefficient matrices $A_i$ of the mNEPv~\eqref{eq:nepv} 
are real symmetric, 
then $H(x)$ is real symmetric and the eigenvectors of the mNEPv 
are all real vectors (up to a unitary scaling).
\Cref{thm:vc}\eqref{i:thm:vc:2} implies that 
the global maximum of the aMax~\eqref{eq:maxf} 
is always achieved at a real vector $x\in\R^n$, namely,
\begin{equation}\label{eq:fxreal}
\max_{x\in\C^n,\, x^Hx=1} \objf(x) 
\quad = 
\max_{x\in\R^n,\, x^Tx=1} \objf(x).
\end{equation}
The two maximizations above are fundamentally
different in nature. The identity holds only due to the
special formulation of $\objf$ and is revealed by~\Cref{thm:vc}. 
We highlight the identity~\eqref{eq:fxreal} because 
many practical optimization problems 
come in the form of the right hand side (with $x\in\R^n$).
We can nevertheless view such a problem as an
aMax~\eqref{eq:maxf} with $x\in\C^n$. 
This allows us to develop a unified treatment for both real and complex
variables, which is highly beneficial as demonstrated for
numerical radius computation in \Cref{eg:quatic}.

\section{Geometry and global convergence of the SCF} \label{sec:scf}


Much like the power method is a gateway algorithm for 
linear eigenvalue problems, 
self-consistent-field, or SCF, 
can be viewed as an analogous algorithm 
for NEPv; see, e.g.,~\cite{Cai:2018,Martin:2004} and references therein.
For the mNEPv~\eqref{eq:nepv}, an SCF iteration starts from an initial unit 
vector $x_0\in\C^n$ and generates 
a sequence of approximate eigenvectors $x_1,x_2\dots,$
via sequentially  solving the linear eigenvalue problems
\begin{equation}\label{eq:scf}
    H(x_k)x_{k+1}  = \lambda_{k+1} \, x_{k+1},
\end{equation}
for $k=0,1,\dots$, where $\lambda_{k+1}$ is the largest eigenvalue of $H(x_k)$
and $x_{k+1}$ is a corresponding unit eigenvector.

In the following, we first present a geometric interpretation of 
the SCF iteration~\eqref{eq:scf}.  Based on the geometric observation, 
we provide a proof of the global convergence of the SCF,
and then we present an acceleration technique and 
discuss related implementation details.

\subsection{Geometry of the SCF}\label{sec:geo}

In~\Cref{sec:variation}, we have discussed the variational characterization
of the mNEPv~\eqref{eq:nepv} via the aMax~\eqref{eq:maxf}.
Consider the change of variables 
\begin{equation} \label{eq:rho2}
y = \rqs(x), 
\end{equation} 
where $\rqs$ is a vector-valued function 
\[ 
\rqs: \R^n \rightarrow \R^m 
\quad \mbox{with} \quad 
\rqs(x): = \left[\,x^H A_1 x,\, \ldots,\, x^H A_m x\,\right]^T.
\] 
The aMax~\eqref{eq:maxf} can then be recast as 
an optimization over the joint numerical range:
\begin{equation}\label{eq:maxphi}
\max_{y \in W(\mathcal A)} 
\Big\{ \phi(y):= \sum_{i=1}^m \phi_i(\,y(i)) \Big\},
\end{equation}
where $y(i)$ is the $i$-th entry of $y$, and 
$W(\mathcal{A})\subset\R^m$ is a (first)
\emph{joint numerical range}
of an $m$-tuple $\mathcal A:= (A_1,\dots,A_m)$ of
Hermitian matrices $A_1, \ldots, A_m$ defined as 
\begin{equation}\label{eq:jointnr2}
W(\mathcal A) = \Big\{\,y  \in \R^m \mid 
y = \rqs(x),\, x \in\mathbb C^m,\, \|x\|=1 \,\Big\}.
\end{equation}
By definition, $W(\mathcal A)$ is the range of 
the vector-valued function $\rqs$ over the unit sphere 
$\{x\in\C^n \mid \|x\|=1\}$.
Since $\rqs$ is a continuous and bounded function, $W(\mathcal A)$ is a connected 
and bounded subset of $\R^m$.
Moreover, it is known that the set of $W(\mathcal A)$ is convex in cases
such as $m=2$ for any $n$, $m=3$ for $n\geq 3$
(see~\cite{Au:1983,Au:1984}),
and other cases under certain conditions~\cite{Li:2000}.

Let $\Omega$ be a bounded and closed subset of $\R^m$, 
and let $v\in\R^m$ be a nonzero vector. Then for a vector 
\begin{equation}\label{eq:constc}
y_{v} \in \argmax_{y\in\Omega} v^Ty,
\end{equation}
the set
\begin{equation}\label{eq:suppomega}
\Big\{\, y\in \R^m \mid  v^T(y-y_{v}) = 0\,\Big\} 
\end{equation}
defines a \emph{supporting hyperplane} of $\Omega$ with 
an outer normal vector $v$ and a supporting point $y_{v}$.
In other words, 
the hyperplane~\eqref{eq:suppomega} contains $\Omega$ in one of its
half-space, and also contains a point $y_{v} \in \Omega$: 
\begin{equation}\label{eq:supp0}
(i)\quad  \mbox{$v^T y\leq  v^Ty_{v}$ for all $y\in \Omega$}
\quad \mbox{and} \quad 
(ii)\quad  \mbox{$y_{v} \in  \Omega$}.
\end{equation}

The following lemma shows that 
if the set $\Omega=W(\mathcal  A)$,
then the optimization~\eqref{eq:constc}, namely the supporting point $y_{v}$,
can be found by solving an Hermitian eigenvalue problem.

\begin{lemma}\label{lem:support}
Let $v\in\R^m$ be a nonzero vector. Then
\begin{equation}\label{eq:ylo}
y_{v} \in  \argmax_{y\in W(\mathcal A)}\ v^Ty
\end{equation}
if and only if 
\begin{equation}\label{eq:supportv} 
y_{v} = \rqs(x_v),
\end{equation}
where $x_v$ is an eigenvector corresponding to
the largest eigenvalue $\lambda_v$ of the Hermitian matrix 
\begin{equation} \label{eq:Hvdef} 
H_v :=\sum_{i=1}^m v(i) \cdot A_i,
\end{equation} 
where $v(i)$ is the $i$-th entry of $v$. 
\end{lemma}
\begin{proof}
	Observe that
	\begin{equation}\label{eq:vtrho}
		v^T \rqs(x) = \sum_{i=1}^m (x^HA_ix) \cdot v(i) = x^HH_vx.
	\end{equation}
	The maximization from~\eqref{eq:ylo} leads to
	\begin{align*}
		\max_{y\in W(\mathcal A)} v^T  y 
		& = \max_{\|x\|=1} v^T  \rqs(x)
		= \max_{\|x\|=1} x^HH_vx  = x_{v}^H H_v x_{v} = v^T\rqs(x_{v}),
	\end{align*}
where the second and the last equalities are due to~\eqref{eq:vtrho},
and the third equality is by the eigenvalue maximization
	principle of Hermitian matrices,
	indicating that the maximizer of $x^HH_vx$
	is achieved at any eigenvector $x_v$ corresponding to
	the largest eigenvalue of $H_v$.
\end{proof}

\Cref{lem:support} suggests a close relation between the SCF
iteration~\eqref{eq:scf} and the search for 
supporting points of $W(\mathcal A)$. Such relation 
is called a geometric interpretation of the SCF, and is formally 
stated in the following theorem.

\begin{theorem}\label{thm:geo}
Let $\{x_k\}$ be the sequence of unit vectors generated by the SCF
iteration~\eqref{eq:scf}, 
and $y_k := \rqs(x_k)$, where $\rqs$ is defined in~\eqref{eq:rho2}. 
Then it holds
\begin{equation}\label{eq:ykp1}
y_{k+1} \in \argmax_{y\in W(\mathcal A)}\ \nabla \phi(y_k)^Ty.
\end{equation}
Therefore, geometrically, 
\begin{equation}\label{eq:scfgeo}
\mbox{$y_{k+1}$ is a supporting point of $W(\mathcal A)$ 
	for the outer normal vector $\nabla \phi (y_k)$}.
\end{equation}
\end{theorem}
\begin{proof}
By the definition~\eqref{eq:hx},
the coefficient matrix $H(x_k)$ is an $H_v$ matrix in~\Cref{lem:support}:
\begin{equation}\label{eq:supph}
	H(x_k) \equiv H_{v_k}
	\quad\text{\quad with $v_k = \nabla \phi(y_k)$
	and  $y_k=\rqs(x_k)\in W(\mathcal A)$}.
\end{equation}
Hence, 
the $k$-th SCF iteration~\eqref{eq:scf}
is to solve the Hermitian eigenproblem
$H_{v_k} x_{k+1} =\lambda_{k+1}x_{k+1}$.
It follows from~\Cref{lem:support} 
that $y_{k+1} = \rqs(x_{k+1}) $ is a solution 
to~\eqref{eq:ylo} for $v_k=\nabla \phi(y_k)$.
Therefore, $y_{k+1}$ is a supporting point of $W(\mathcal A)$ 
for the outer normal direction $\nabla \phi(y_k)$.
\end{proof}

At a solution $x_*$ of the mNEPv~\eqref{eq:nepv}, 
the geometric interpretation~\eqref{eq:scfgeo} is equivalent to
the following geometric first-order optimality condition 
for the constrained optimization~\eqref{eq:maxphi}: 
\begin{equation}\label{eq:geonepv}
\nabla \phi(y_*)  \text{ is an outer normal vector of $W(\mathcal A)$ at $y_*$},
\end{equation}
where $y_* = \rqs(x_*)$.
In fact, the statement~\eqref{eq:geonepv} can be obtained 
via the geometric viewpoint of the optimality condition 
for a general constrained optimization,
described by the \emph{normal cone} of the feasible set; 
see, e.g.,~\cite[Sec. 12.7]{Nocedal:2006}.

By~\Cref{thm:geo}, the SCF iteration~\eqref{eq:scf} can be visualized
as searching the solution of the mNEPv~\eqref{eq:nepv}
on the boundary of the joint numerical range $W(\mathcal{A})$. 
For illustration, let us consider the mNEPv~\eqref{eq:nepv} of the form
\begin{equation}\label{eq:hxqrt0}
H(x)x=\lambda x 
\quad\text{with}\quad
H(x) = (x^HA_1x)\cdot A_1 +  (x^HA_2x)\cdot A_2,
\end{equation}
where $A_1$ and $A_2$ are Hermitian matrices.
The mNEPv~\eqref{eq:hxqrt0} arises from numerical radius computation
and will be further discussed in \Cref{eg:quatic}. 
By~\Cref{thm:vc} and~\eqref{eq:maxphi},
the mNEPv~\eqref{eq:hxqrt0} can be 
characterized by the following optimization problems:
\begin{equation}\label{eq:numrdmax0}
\max_{x\in\C^n, \|x\|=1} 
\left\{ \objf(x) := \frac{1}{2}\left( (x^HA_1x)^2+ (x^HA_2x)^2 \right) \right\} 
=
\max_{y\in W(A_1,A_2)}
\left\{\, \phi(y) := \frac{1}{2}\|y\|^2\, \right\},
\end{equation}
where $W(A_1,A_2)$ is a joint numerical range of $A_1$ and $A_2$. 
\Cref{fig:geo} depicts the SCF as a search process 
for solving the mNEPv~\eqref{eq:hxqrt0} with random Hermitian 
matrices $A_1$ and $A_2$ of order 10. Given the initial $y_0=\rqs(x_0)$,
the SCF first searches in the gradient direction $v_0=\nabla \phi(y_0)$ 
to obtain a supporting point $y_1=\rqs(x_1)$;
it then searches in the gradient direction $\nabla \phi(y_1)$ to  obtain
the second supporting point $y_2=\rqs(x_2)$; and so on. 
When this process converges to $y_*=\rqs(x_*)$, the gradient $\nabla \phi(y_*)$ 
overlaps the outer normal vector for $W(\mathcal A)$ at $y_*$,
i.e., the geometric optimality condition~\eqref{eq:geonepv} is achieved.

\begin{figure}[t]
\centering\includegraphics[width=0.5\textwidth]{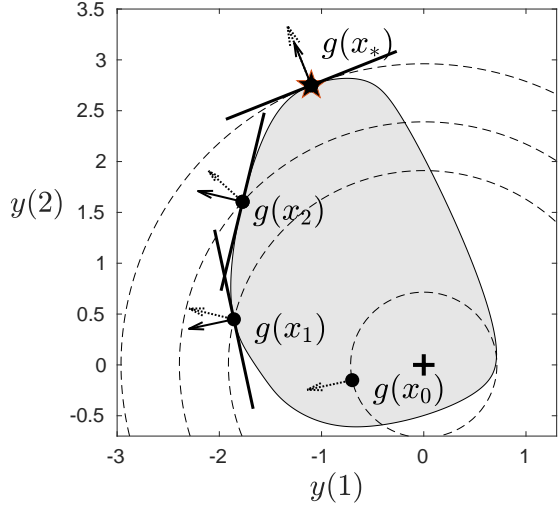}
\caption{ 
Illustration of the first three iterates $x_0,x_1,x_2$ 
by the SCF~\eqref{eq:scf} for the mNEPv~\eqref{eq:hxqrt0}: 
the dark region is the joint numerical range $W(A_1,A_2)$;
dashed lines represent contour of $\phi(y)=\|y\|^2/2$ 
with dashed arrows the gradient directions $\nabla \phi$;
solid tangent lines are `supporting hyperplanes'
at $y_i=\rqs(x_i)$; the maximizer of~\eqref{eq:numrdmax0} is marked as 
$\text{\FiveStar}$.
}\label{fig:geo}
\end{figure}

Let us discuss two direct implications of~\Cref{thm:geo}. 
Firstly, equation~\eqref{eq:ykp1}
indicates that the SCF is a \emph{successive local linearization} for the
optimization~\eqref{eq:maxphi}:
At iteration $k$, we approximate the function
$\phi(y)$ by a first-order expansion 
\begin{equation}\label{eq:ellky}
\ell_k(y):=\phi(y_k) + \nabla\phi(y_k)^T(y-y_k)
\end{equation}
and then solve the optimization of the linear function over the joint
numerical range as
\begin{equation}\label{eq:locallin}
\max_{y\in W(\mathcal A)} \ell_k(y).
\end{equation}
By dropping the constant terms in $\ell_k(y)$, the maximizers
of~\eqref{eq:locallin} satisfies
\[
\argmax_{y\in W(\mathcal A)}\ \ell_k(y) \equiv \argmax_{y\in W(\mathcal A)}\
 \nabla\phi(y_k)^Ty.
\]
Hence, the solution
to~\eqref{eq:locallin}
is exactly $y_{k+1}$ in~\eqref{eq:ykp1}, and we have
\begin{equation}\label{eq:lkp1}
\ell_k(y_{k+1}) = \max_{y\in W(\mathcal A)} \ell_k(y).
\end{equation}
These observations are helpful to the proof of the global convergence of
the SCF 
as to be presented in \Cref{sec:scfconv}. 
Secondly, it is well-known that 
a closed convex region $\Omega$ is the intersection of all of its 
supporting halfspaces.
One can use intersections of sampled supporting halfspaces (i.e., a polytope)
to approximate $\Omega$ from outside; see, e.g.,~\cite[Sec.~11]{Rockafellar:1970}.
Such schemes,  known as \emph{outer approximation},
are commonly used for finding global optimizers of
convex maximization over a convex region; see, e.g.,~\cite{Benson:1995}
for a general description and~\cite{Uhlig:2009} for 
a geometric computation of numerical radius.
Since the SCF generates a sequence of supporting hyperplanes of
$W(\mathcal A)$, 
those hyperplanes also produce outer approximations of
$W(\mathcal A)$, 
which allows to combine the SCF with outer approximation schemes 
for the global optimization of~\eqref{eq:maxphi}. Detailed discussion of 
such approach is beyond the scope of this paper.


\subsection{Convergence analysis of the SCF}\label{sec:scfconv}
In this section, we show that the SCF iteration is 
globally convergent to an eigenvector of the mNEPv~\eqref{eq:nepv}, 
as indicated by the visualization of the SCF in~\Cref{sec:geo}.
Moreover, the converged eigenvector is typically a stable one and
the rate of convergence is at least linear.

Let $\{x_k\}$ be a sequence of unit vectors.
We call $x_*$ an (entry-wise) limit point of $\{x_k\}$ if
\begin{equation}\label{eq:liminting}
\text{ 
$x_* = \lim_{j\to\infty} x_{k_j}$ for some subsequence $\{x_{k_j}\}$
indexed by $k_1 < k_2<\cdots$.
}
\end{equation}
By the well-known Bolzano–Weierstrass theorem,
a bounded sequence in $\C^n$ always has a convergent subsequence. 
So the sequence $\{x_k\}$ of unit vectors 
has at least one limit point $x_*$.

The following theorem shows the global convergence of 
the SCF iteration~\eqref{eq:scf} .

\begin{theorem}\label{thm:mono}
Let $\{x_k\}$ be the sequence of unit vectors generated by 
the SCF~\eqref{eq:scf} for the mNEPv~\eqref{eq:nepv}, 
and let $\objf(x)$ be the 
objective function of the corresponding aMax~\eqref{eq:maxf}.  Then, 
\begin{enumerate}[(a)]
\item\label{thm:mono:mono}
$\objf(x_{k+1})\geq \objf(x_k)$ for $k=0,1,\dots$,
with equality holds only if $x_k$ 
is an eigenvector of the mNEPv~\eqref{eq:nepv}; 

\item\label{thm:mono:nepv} 
each (entry-wise) limit point $x_*$ of $\{x_k\}$ 
must be an eigenvector of the mNEPv~\eqref{eq:nepv},
and it holds $\objf(x_*) \geq \objf(x_k)$ for all $k\geq 0$.
\end{enumerate}
\end{theorem} 
\begin{proof} 
For item~\eqref{thm:mono:mono}, the monotonicity 
$\objf(x_{k+1}) \geq \objf(x_k)$  
is a direct consequence of the convexity of $\phi$
and~\eqref{eq:lkp1}.
Recalling~\eqref{eq:ellky} that the linearization $\ell_k$
of the convex function $\phi$
is always a lower supporting function of $\phi$,  i.e., 
\[
	\text{$\ell_k(y) \leq \phi(y)$ for all $y\in W(\mathcal A)$,}
\]
we have 
\begin{equation}\label{eq:phikp1}
\objf(x_{k+1})\equiv \phi(y_{k+1}) \geq \ell_k(y_{k+1}) = \max_{y\in W(\mathcal A)}
\ell_k(y) \geq \ell_k(y_k) = \phi(y_k) \equiv \objf(x_k),
\end{equation}
where the third equality is by~\eqref{eq:lkp1}.
Moreover, if the equality $\objf(x_{k+1})=\objf(x_k)$ holds, 
then~\eqref{eq:phikp1} implies 
\begin{equation}\label{eq:yklk}
	\ell_k(y_k) = \max_{y\in W(\mathcal A)} \ell_k(y),
\end{equation}
namely,
\[
	y_{k}\in
	\argmax_{y\in W(\mathcal A)} \ell_k(y)
	\equiv 
	\argmax_{y\in W(\mathcal A)} \nabla\phi(y_k)^Ty.
\]
According to~\Cref{lem:support},
$y_k = \rqs(x_k)$ and $x_k$ is an eigenvector for the largest eigenvalue
of $H_{v_k}$ with $v_k=\nabla \phi(y_k)$.
Since $H_{v_k}\equiv H(x_k)$, we have 
$H(x_k)x_k = \lambda x_k$ and $\lambda$ is the largest eigenvalue,
i.e., $x_k$ is an eigenvector of the mNEPv~\eqref{eq:nepv}.

For item~\eqref{thm:mono:nepv}, 
let $\{x_{k_j}\}$ be a subsequence of $\{x_k\}$ convergent to $x_*$.
The monotonicity of $\objf(x_{k+1})\geq \objf(x_k)$ from item~\eqref{thm:mono:mono} implies 
$\objf(x_*) \geq \objf(x_k)$ for all $k\geq 0$.

To show $x_*$ is an eigenvector, we denote by $y_{k_j}= \rqs(x_{k_j})$
and $y_{*}= \rqs(x_{*})$. The linearization of $\phi$ at $y_*$ satisfies 
\begin{equation}\label{eq:convl}
\ell_*(y):=\phi(y_*) + \nabla\phi(y_*)^T(y-y_*)
= \lim_{j\to \infty} \ell_{k_j} (y),
\end{equation}
where the last equality is due to $y_*=\lim_{j\to\infty} y_{k_j}$
and  the continuity of $\phi$ and $\nabla\phi$
(recall $\ell_k$ from~\eqref{eq:yklk}).

We first show that 
\begin{equation}\label{eq:nabphi}
\nabla\phi(y_*)^T(y-y_*) \leq 0 
\quad \text{for all $y\in W(\mathcal A)$}.
\end{equation}
Otherwise, there exists a $\widetilde y\in W(\mathcal A)$ with
\begin{equation}\label{eq:vareps}
	\varepsilon := \nabla\phi(y_*)^T(\widetilde y-y_*) > 0.
\end{equation}
Due to the convergence of $\ell_{k_j}\to \ell_*$ in~\eqref{eq:convl},
there exists $N\geq 0$ such that for all $j\geq N$,
\begin{equation}\label{eq:lkj}
\ell_{k_j}(\widetilde y)  \geq \ell_*(\widetilde y) -
{\varepsilon}/{2}.
\end{equation}
It then follows from~\eqref{eq:phikp1} (with $k=k_{j}$) that
for all $j\geq N$,
\[
	\phi(y_{k_j+1}) \geq \max_{y\in W(\mathcal A)} \ell_{k_j}(y)
	\ \geq \
	\ell_{k_j}(\widetilde y)\ \geq\ \ell_*(\widetilde y) - \frac{\varepsilon}{2}
	= \phi(y_*) + \frac{\varepsilon}{2},
\]
where the last two equations are due to~\eqref{eq:lkj}
and~\eqref{eq:vareps}. 
This implies $\objf(x_{k_j+1}) \geq \objf(x_*) + \varepsilon/2$,
contracting $\objf(x_*)\geq \objf(x_k)$ for all $k$.

It follows from~\eqref{eq:convl} and~\eqref{eq:nabphi} that
\[
\ell_*(y_*) =  \max_{y\in W(\mathcal A)} \ell_*(y) =  \phi(y_*).
\]
By the same arguments as for the $y_k$ in~\eqref{eq:yklk},
we have $x_*$ is an eigenvector of the mNEPv~\eqref{eq:nepv}.
\end{proof}

In~\Cref{sec:nepvegs},
we will discuss the mNEPv~\eqref{eq:nepv} 
arising from optimization in the form of the aMax~\eqref{eq:maxf}.
Then the monotonicity of the objective function 
is highly desirable. 
Starting from any $x_0$, the SCF will find an eigenvector $x_*$ that has a 
increased function value $\objf(x_*) \geq \objf(x_0)$.

\Cref{thm:mono} guarantees the global convergence of the SCF to an eigenvector
$x_*$ of the mNEPv~\eqref{eq:nepv}. 
It may happen that $x_*$ is a non-stable eigenvector.
For example, if the initial $x_0$ is itself a non-stable eigenvector,
or if the SCF unluckily jumps to an exact non-stable eigenvector during the
iteration, then the iteration stagnates at that eigenvector.
However except those special situations, 
the convergence to a non-stable eigenvector rarely happens in practice.
This is because the SCF~\eqref{eq:scf} is a fixed-point iteration
of the mapping $\Pi$~\eqref{eq:pi}. 
$\Pi$ is non-contractive at non-stable eigenvectors; see~\Cref{sec:stability}. 
By the local convergence analysis of the SCF 
for a general unitarily invariant NEPv (see~\cite[Theorem~1]{Bai:2022}), 
we can draw the local convergence of the SCF iteration~\eqref{eq:scf} 
for the mNEPv~\eqref{eq:nepv} as stated in the following theorem.

\begin{theorem}\label{thm:bailocal}
Let $x_*$ be an eigenvector of the mNEPv~\eqref{eq:nepv}
corresponding to a simple eigenvalue $\lambda_*$,
$\mathcal L$ be the $\mathbb R$-linear operator~\eqref{eq:lop1}
with respect to $x_*$, and $\spr(\mathcal L)$ be the spectral radius 
of $\mathcal{L}$.
\begin{enumerate}[(a)]
\item\label{thm:bailocal:conv}
If $\spr(\mathcal L) <1$ 
(i.e., $x_*$ is a stable eigenvector by~\Cref{def:regular2}),
then the SCF~\eqref{eq:scf} is locally convergent to
$x_*$, with an asymptotic convergence rate bounded by 
$\spr(\mathcal L)$.

\item\label{thm:bailocal:div}
If $\spr(\mathcal L) >1$ 
(i.e., $x_*$ is a non-stable eigenvector by~\Cref{def:regular2}),
then the SCF is locally divergent from $x_*$.
\end{enumerate}
\end{theorem}
Here we recall that an iterate $x_{k}$ by the SCF~\eqref{eq:scf} 
is understood as an one-dimensional subspace spanned by $x_{k}$. 
The local convergence and divergence of $x_k$ in~\Cref{thm:bailocal} 
is measured by the vector angle 
$\angle(x_*,x_k):= \cos^{-1}\left(|x_*^Hx_k|\right)$.

\begin{figure}[t]
\centering\includegraphics[width=0.5\textwidth]{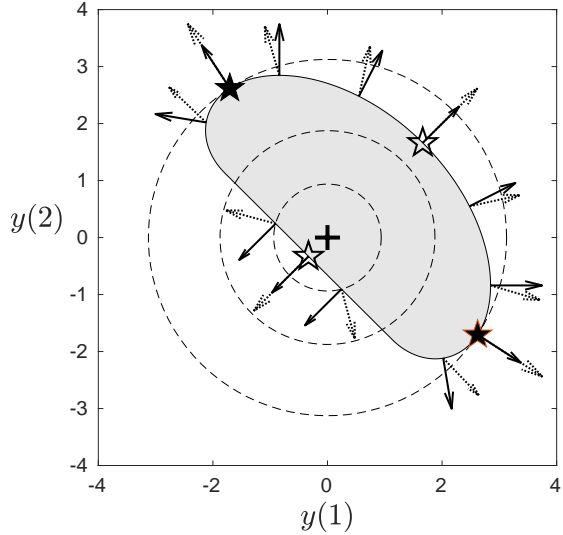}
\caption{Illustration of stable eigenvectors, marked as solid stars
$\text{\FiveStar}$,
and non-stable eigenvectors, marked as hollow stars $\text{\FiveStarOpen}$:
Close to a non-stable eigenvector, the gradients $\nabla\phi$ (dashed arrows)
point away from the normal vectors (solid arrow),
leading to divergence of the SCF from $\text{\FiveStarOpen}$.
}\label{fig:geo2}
\end{figure}

By the geometric interpretation of the SCF from \Cref{thm:geo}, 
we can visually illustrate its local convergence behavior 
revealed in~\Cref{thm:bailocal}.
\Cref{fig:geo2} depicts the search directions of the SCF
for the numerical radius problem described
in~\eqref{eq:numrdmax0}, with the corresponding mNEPv~\eqref{eq:hxqrt0}. 
There are four eigenvectors (marked as stars, where the solid
and dashed arrows overlap). Two solid stars are 
stable eigenvectors (i.e., local maximizers of~\eqref{eq:numrdmax0})
and two hollow stars are non-stable eigenvectors (non-maximizers). 
The reason why the SCF is locally convergent to stable eigenvectors
is now clear: Close to a solid star,
the search directions $\nabla \phi(y)$ by~\eqref{eq:scfgeo} (dashed arrow) 
leads the next iteration closer to the solid star.
In contrast, close to a hollow star, 
the search directions lead away from the hollow star.
This observation also justifies the name of \emph{non-stable
eigenvector}, since a slight perturbation 
will lead the SCF to diverge from those solutions.

Combining the properties of global convergence in \Cref{thm:mono}
and the local convergence in \Cref{thm:bailocal},
we can summarize the overall convergence behavior of the
SCF~\eqref{eq:scf} as follows:
\begin{enumerate} 
\item 
Let $x_*$ be an (entry-wise)
limit point of $\{x_k\}$ by the SCF. 
Then $x_*$ is an eigenvector of the mNEPv~\eqref{eq:nepv};
see~\Cref{thm:mono}\eqref{thm:mono:nepv}. 

\item 
The limit point $x_*$ is \emph{unlikely} a non-stable eigenvector,
since the SCF is locally divergent from non-stable eigenvectors;
see~\Cref{thm:bailocal}\eqref{thm:bailocal:div}.\footnote{One 
exceptional but rare case is that some $x_k$ coincides with a
non-stable eigenvector and SCF stops at $x_k$.}
Consequently, the SCF is expected to converge to (at least) 
a weakly stable eigenvector $x_*$ of the mNEPv~\eqref{eq:nepv}.

\item 
If the limit point $x_*$ is a stable eigenvector, then 
the SCF is at least locally linearly convergent to $x_*$;
see~\Cref{thm:bailocal}\eqref{thm:bailocal:conv}.
\end{enumerate}

\subsection{Accelerated SCF}\label{sec:impl}
The iterative process~\eqref{eq:scf} is an SCF in its simplest form, 
also known as the plain and pure SCF.
There are a number of ways to accelerate the plain SCF,
such as the damping scheme~\cite{Cances:2000b},
level-shifting~\cite{Saunders:1973,Thogersen:2004,Yang:2007},
direct inversion of iterative subspace (DIIS) 
with Anderson acceleration~\cite{Pulay:1980,Pulay:1982}, 
and preconditioned fixed-point iteration; see, e.g.,~\cite{Lin:2013}.
Most of these schemes are originally designed for solving NEPv 
from electronic structure calculations.  
In this section, we present an acceleration scheme 
of the SCF~\eqref{eq:scf} for the mNEPv~\eqref{eq:nepv}.

Inverse iterations are commonly applied for 
linear eigenvalue problems~\cite{Ipsen:1997} and
eigenvalue-dependent nonlinear eigenvalue problems (NEP)~\cite{Guttel:2017}.
There is an inverse iteration~\cite{Jarlebring:2014} for NEPv 
in the form of 
\begin{equation} \label{eq:nepv3} 
H(x/\|x\|)\cdot x=\lambda x, 
\end{equation} 
where $H(x)$ is
a real symmetric matrix that is differentiable in $x\in\R^n$~\cite{Jarlebring:2014}.\footnote{
	The authors in~\cite{Jarlebring:2014} considered scaling invariant NEPv 
	$H(x)\cdot x=\lambda x$ with $H(x)\equiv H(\alpha x)$ for all $\alpha\neq 0$,
	and they pointed out such NEPv cover~\eqref{eq:nepv3} as a special case.
}
For normalized $x$, we have $H(x/\|x\|)\equiv H(x)$, 
so that the mNEPv~\eqref{eq:nepv} can be equivalently 
written to an NEPv~\eqref{eq:nepv3}.
In the following, we will first revisit the inverse iteration scheme
in~\cite{Jarlebring:2014},
and then propose an improved scheme for solving 
the mNEPv~\eqref{eq:nepv} by exploiting its underlying structure. 

Let $x_k$ be a unit approximate eigenvector of 
the NEPv~\eqref{eq:nepv3}
and $\sigma_k$ be a chosen shift close to a target eigenvalue.
The following inversion step is proposed in~\cite{Jarlebring:2014} to
improve $x_k$:
\begin{equation}\label{eq:inverse}
\widetilde x_{k} = \alpha_k \left(J(x_k) - \sigma_k I\right)^{-1} x_k
\quad\text{with\quad $J(x) := \frac{\partial}{\partial x} (H(x/\|x\|)x)$},
\end{equation}
where $\alpha_k$ 
is a normalization factor.
In~\cite{Jarlebring:2014}, it is shown that 
the iteration~\eqref{eq:inverse} is closely related to Newton's method
for the nonlinear equations $H(x/\|x\|)x-\lambda x = 0$ and $x^Tx=1$.
Moreover, under mild assumptions, 
an inverse iteration that recursively applies~\eqref{eq:inverse}
converges linearly, 
with a convergence factor proportional to the
distance between the shift $\sigma_k$ and the target eigenvalue.
Furthermore, a quadratic convergence is expected with the Rayleigh shift 
$\sigma_k = x_k^T H(x_k)x_k$.

Directly applying the inverse iteration~\eqref{eq:inverse} 
to the mNEPv~\eqref{eq:nepv} will ignore the requirement that $\lambda$
is the largest eigenvalue of $H(x)$. Consequently, the process is prone to
convergence to an eigenvalue that is not the largest one. 
Nevertheless, the rapid quadratic convergence of the inverse iteration
with Rayleigh shifts is appealing.
We hence propose to only use the inversion step~\eqref{eq:inverse} 
as an acceleration for the SCF iteration~\eqref{eq:scf}.

We first note that despite the matrix $H(x)$ 
of the mNEPv~\eqref{eq:nepv} is symmetric 
when all coefficient matrices $A_1,\dots,A_m$ are real symmetric, 
the Jacobian~$J(x)$ in~\eqref{eq:jx} is generally not. 
Specifically, the Jacobian $J(x)$ of $H(x)$ is given by
\begin{align} \label{eq:jx} 
J(x) &\equiv \frac{\partial}{\partial x} (\,H(x/\|x\|)x\,) 
= H(x) + 2 M(x)C(x)M(x)^T P(x),
\end{align}
where 
$M(x)= [A_1x,\, \cdots,\, A_mx]$ and 
$C(x)=\mbox{Diag}\left(
h_1'\left( x^TA_1 x\right),\dots,h_m' \left(x^TA_m x\right)\right)$,
and $P(x) = I-xx^T$ is a projection matrix.
To symmetrize $J(x)$, we introduce 
\begin{equation}\label{eq:jsx}
J_{\rm s}(x):= H(x) + 2\, P(x)M(x)C(x)M(x)^T P(x),
\end{equation}
and reformulate the iteration step~\eqref{eq:inverse} to
\begin{equation}\label{eq:syminverse}
\widetilde x_{k} = 
\widetilde \alpha_k \cdot \left(J_{\rm s}(x_k) - \sigma_k I\right)^{-1} x_k,
\end{equation}
where $\widetilde \alpha_k$ normalizes $\widetilde x_k$ to a unit vector.
We observe that the iterations~\eqref{eq:inverse} and~\eqref{eq:syminverse}
are equivalent in the sense that replacing $J(x)$ by $J_s(x)$ 
in~\eqref{eq:inverse} will not affect the direction of $\widetilde x$. 
This is due to the fact that  
$$J_{\rm s}(x)  = J(x) + x\cdot q(x)^T,$$
where $q(x) = 2 P(x)M(x)C(x)M(x)^Tx\in\R^n$. 
Then by the Sherman–Morrison-Woodbury formula \cite{Horn:2012},
a quick algebraic manipulation shows that
\[
 \left(J_{\rm s}(x_k) - \sigma_k I\right)^{-1} x_k
 = 
 c\cdot \left(J(x_k) - \sigma_k I\right)^{-1} x_k
\]
for some some constant $c$.

If the coefficient matrices $A_1,\dots,A_m$ 
are complex Hermitian, then $H(x)$ is not holomorphically differentiable
(since the diagonal entries of $H(x)$ are always real and cannot be 
analytic functions).  In this case, the matrix $J_{\rm s}(x)$ in~\eqref{eq:jsx}
no longer corresponds to the (holomorphic) Jacobian of $H(x/\|x\|)x$.
Nevertheless, $J_{\rm s}(x)$ is well-defined and Hermitian 
(replacing all the transpose $\cdot^T$ to conjugate transpose $\cdot^H$).
We can still use it for the iteration~\eqref{eq:syminverse}.

The SCF with an optional acceleration for solving the mNEPv~\eqref{eq:nepv}
is summarized in~\Cref{alg:scflf}.
A few remarks on the implementation detail are in order.

\begin{algorithm}[tbhp]
\caption{The SCF with optional acceleration}\label{alg:scflf}
\begin{algorithmic}[1]
\REQUIRE Starting vector $x_0\in\C^n$,
residual tolerance $\mbox{tol}$, and 
acceleration threshold $\mbox{tol}_{\rm acc}$.

\ENSURE Approximate eigenpair $(\lambda_k,x_k)$ of the mNEPv~\eqref{eq:nepv}.

\FOR{$k=1,2,\dots$}

\STATE \label{alg:scflf:scf}
$H(x_{k-1})\, x_{k}  = \lambda_{k} \cdot x_{k}$ with 
$\lambda_{k} =\lambda_{\max}(H(x_{k-1}))$;  
\COMMENT{\hfill \% SCF}

\STATE 
{\bf if} $\res(x_{k})\leq \mbox{tol}$, {\bf then} return 
$(\lambda_{k},x_{k})$; 
\COMMENT{\hfill \% test for convergence}
\label{alg:scflf:convtst}

\IF[\hfill \%  acceleration if activated]
{$\res(x_{k})\leq \mbox{tol}_{\rm acc}$}
\label{alg:scflf:acc}  

\STATE 
compute $\widetilde x_{k}$ by~\eqref{eq:syminverse} 
with the shift $\sigma_k = x^H_k H(x_k) x_k$.
\label{alg:scflf:ref}

\STATE 
{\bf if} $\objf(\widetilde x_{k})>\objf(x_{k})$, {\bf then} update $x_{k} = \widetilde x_{k}$;
\label{alg:scflf:mono}

\ENDIF 
\ENDFOR
\end{algorithmic}
\end{algorithm} 

\begin{enumerate}
\item
The initial $x_0$, in view of the geometry of the SCF discussed in~\Cref{sec:geo},
can be chosen from sampled supporting points of $W(\mathcal A)$.

Let $v_i\in\R^m$, for $i=1,\dots, \ell$, be 
randomly sampled search directions.
We can first find the supporting points $y_{v_i} = \rqs(x_{v_i})$ of 
$W(\mathcal A)$,
and then among $x_{v_1},\dots, x_{v_\ell}$, take the one with the largest
value $ \objf(x_{v_i})$ as the initial $x_0$. 
This greedy sampling schme increases the chance for 
the SCF to find the global maximizer of the aMax~\eqref{eq:maxf}.

For computation, recall~\Cref{lem:support} that each $x_{v_i}$ is 
an eigenvector corresponding to the largest eigenvalue of 
the Hermitian matrix $H_{v_i}$ defined in \eqref{eq:Hvdef}. 
Therefore, it requires to solve $\ell$ Hermitian eigenvalue problems
for sampling $\ell$ supporting points. For efficiency, 
we can exploit the fact that the smallest eigenvalue of $H_{v_i}$ 
is corresponding to the largest of $H_{-v_i} \equiv -H(v_i)$.
Hence, we can compute two supporting points in both directions $\pm v_i$
by solving a single eigenvalue problem of $H_{v_i}$. 

\item
\Cref{alg:scflf} requires 
finding the eigenvector corresponding to the largest eigenvalue
of the matrix $H(x_{k-1})$ in line~\ref{alg:scflf:scf}. 
In addition, when the acceleration is applied, a solution of linear system 
with coefficient matrix $J_{\rm s}(x_k) - \sigma_k I$  in line~\ref{alg:scflf:ref}.
For the mNEPv of small to medium sizes, direct solvers can be applied, 
such as QR algorithm for Hermitian eigenproblems and  LU factoration 
for linear systems.
For large sparse problems, iterative solvers are applied,
e.g., the Lanczos type methods for Hermitian eigenproblems 
(such as MATLAB \text{eigs}), and MINRES and SYMMLQ 
for linear systems; see, e.g.,~\cite{Bai:2000, Barrett:1994}.


\item 
The acceleration with the inverse iteration is expected to 
work well for $x_k$ close to a solution. 
We introduced a threshold $\mbox{tol}_{\rm acc}$ to control 
the activation of inverse iteration in line~\ref{alg:scflf:acc}.   
If $\mbox{tol}_{\rm acc}=0$, \Cref{alg:scflf} runs the SCF (no
acceleration). If $\mbox{tol}_{\rm acc}=\infty$, 
\Cref{alg:scflf} applies acceleration at each step.
We observe that the choice of $\mbox{tol}_{\rm acc}$ is 
usually not critical 
and $\mbox{tol}_{\rm acc}=0.1$ is used as a default value in our 
numerical experiments.

\item 
To maintain the monotonicity of $\objf(x_k)$, as in the SCF,
the accelerated eigenvector $\widetilde x_{k}$ is accepted only if 
$\objf(\widetilde x_{k})\geq \objf(x_{k})$ in line~\ref{alg:scflf:mono}. 


\item
We use the relative residual norm
\begin{equation}\label{eq:resd}
 \res (\widehat{x}) : = 
  \frac{\|H(\widehat{x}) \widehat{x} - 
(\widehat{x}^H H(\widehat{x})\widehat{x})\cdot 
\widehat{x}\|}{\|H(\widehat{x})\|}
\end{equation}
to assess the accuracy of an approximate eigenvector $\widehat{x}$ 
in line \ref{alg:scflf:convtst},
where $\|H(\widehat{x})\|$ is some convenient to evaluate matrix norm,
e.g., the matrix $1$-norm as we used in the experiments.


\end{enumerate}


\section{Applications} \label{sec:nepvegs} 
The mNEPv~\eqref{eq:nepv} and the associated 
aMax~\eqref{eq:maxf} can be found in numerous applications.
In this section, we discuss three of them. 
The first one is known as the quartic maximization 
over the Euclidean sphere and its applications for computing
the numerical radius. The second one is on the best rank-one
approximation of third order partial-symmetric tensors. 
The third is from the study of the distance to singularity of 
dissipative Hamiltonian differential-algebraic equation systems, or
dHDAE.

\subsection{Quartic maximization and numerical radius} \label{eg:quatic}  
A {\em (homogeneous) quartic maximization} over the Euclidean sphere 
is of the form
\begin{equation}\label{eq:fxqrt}
\max_{x\in\C^n,\, \|x\|=1} 
\Big\{\objf(x) := \frac{1}{2}\sum_{i=1}^{m}\left(x^HA_ix\right)^2\Big\}, 
\end{equation}
where $A_i$ are $n$-by-$n$ Hermitian matrices. 
The optimization~\eqref{eq:fxqrt} is a classical problem in 
the field of polynomial optimization, although in the literature 
it is usually formulated in real variables, 
i.e., $x\in\R^n$ with symmetric $A_i$~\cite{He:2010,Nesterov:2003,Zhang:2012}. 
In addition, such problems also arise in the study of
robust optimization with ellipsoid uncertainty~\cite{Ben:1998}.
Observe that the quartic maximization~\eqref{eq:fxqrt} is 
an aMax~\eqref{eq:maxf} with
$\phi_i(t) = t^2/2$. Hence the underlying mNEPv~\eqref{eq:nepv} is 
of the form
\begin{equation}\label{eq:hxqrt}
H(x)\, x=\lambda x 
\quad\text{with}\quad
H(x) = \sum_{i=1}^m (x^HA_ix)\cdot A_i,
\end{equation}
where the coefficient functions $h_i(t)=\phi_i'(t)=t$ 
are differentiable and non-decreasing. 

The most simple but non-trivial example of 
the quartic optimization~\eqref{eq:fxqrt} 
is $m=2$. It is the well-known problem of
computing the numerical radius of a square matrix.
The {\em numerical radius} 
of a matrix $B\in\C^{n\times n}$ is defined as
\begin{equation}\label{eq:rb}
r(B) := \max_{x\in\C^n,\, \|x\|=1} |x^HBx|
= \max_{x\in\C^n,\, \|x\|=1} \Big((x^HA_1x)^2+ (x^HA_2x)^2\,\Big)^{1/2}, 
\end{equation}
where $A_1 = \frac{1}{2}(B^H+B)$ and $A_2 = \frac{\imath}{2}(B^H-B)$, both
are Hermitian matrices~\cite{Horn:2012}. 
An extension of~\eqref{eq:rb} is 
the {\em joint numerical radius} of 
an $m$-tuple of Hermitian matrices $\mathcal A= (A_1,\dots,A_m)$
defined as 
\begin{equation}\label{eq:jrd}
r(\mathcal A) := \max_{x\in\C^n,\, \|x\|=1} 
\left(\sum_{i=1}^m(x^HA_ix)^2\right)^{1/2},
\end{equation}
see~\cite{Cho:1981}.
The (joint) numerical radius plays important roles in numerical analysis.
For examples, the numerical radius of a matrix  
is applied to quantify the transient effects of
discrete-time dynamical systems and 
analyze classical iterative methods~\cite{Axelsson:1996,Trefethen:2005}.
The joint numerical radius of a matrix tuple 
is used for studying the joint behavior of several operators; 
see~\cite{Li:2012} and references therein.

Numerical algorithms for computing the numerical radius of a single matrix 
have been extensively 
studied~\cite{He:1997,Mengi:2005,Mitchell:2020,Uhlig:2009,Watson:1996}.
To find the global maximizer of~\eqref{eq:rb},
many methods adopt the scheme of local optimization followed 
by global certification. Most of those algorithms, however, 
do not immediately extend to computing the joint numerical radius 
with $m\geq 3$, and neither do they exploit the connection with the NEPv as developed 
in this paper. 
As a major benefit of the NEPv approach, it allows for fast
computation of the local maximizers of the problems, so it
can be used to accelerate existing approaches.
Moreover, the NEPv approach provides a unified treatment for 
matrix tuple $\mathcal A$ with $m$ matrices and, hence, 
can serve as the basis for future development of algorithms
towards the global solution of $r(\mathcal A)$ with $m\geq 3$.

%

\subsection{Best rank-one approximation of third order partial-symmetric tensors}
\label{sec:tensor}  

Let $T\in\R^{n\times n\times m}$ be a third order
partial-symmetric tensor, i.e., $A_i:=T(:,:,i)\in\R^{n\times n}$ 
is symmetric for $i=1,\dots,m$. 
The problem of the best rank-one partial-symmetric tensor 
approximation is defined by the minimization 
\begin{equation} \label{eq:tensor} 
\min_{\mu \in \R,\, x\in\R^n,\, z\in\R^m\atop \|x\|=1,\|z\|=1}
\|T - \mu\cdot x\otimes x\otimes z\|_F^2,
\end{equation} 
where $\otimes$ denotes the Kronecker product. 
The solution of \eqref{eq:tensor} provides a 
rank-one partial-symmetric tensor 
$\mu_*\cdot x_*\otimes x_*\otimes z_*$ 
that best approximates $T$ in the Frobenius norm $\|\cdot\|_F$
and is also known as a truncated rank-one CP decomposition of $T$; 
see, e.g.,~\cite{Kolda:2009,Zhang:2012}.

The best rank-one approximation~\eqref{eq:tensor} are 
often recast as a quartic
maximization~\eqref{eq:fxqrt}; see, e.g.,~\cite[Eq.~(6)]{Chen:2021}.
Let $x_i$ denote the $i$-th element of a vector $x$. Then  
\begin{equation} \label{eq:squarefun} 
\|T - \mu\cdot x\otimes x\otimes z\|_F^2 
 = \|T\|_F^2 + \mu^2 - 2\mu \sum_{i,j,k}t_{ijk}  x_i x_j z_k, 
\end{equation}
where the range of indices $i,j,k$ are omitted in the summation for clarity. 
Since 
the minimum w.r.t. $\mu$ is achieved at 
$\mu  = \sum_{i,j,k}t_{ijk} x_i x_j z_k$,
the best rank-one approximation~\eqref{eq:tensor} becomes the maximization 
\begin{equation}\label{eq:tquartic}
\max_{\|x\|=1\atop \|z\|=1}
\bigg(\sum_{i,j,k}t_{ijk} x_ix_j z_k\bigg)^2
= 
\max_{\|x\|=1\atop \|z\|=1}
\bigg(\sum_{k}z_k\cdot x^TA_kx\bigg)^2
= 
\max_{\|x\|=1}
\sum_{k} \left(x^TA_kx\right)^2,
\end{equation}
where the first equality is by 
$A_i=T(:,:,i)$, and the second equality is 
due to the maximization w.r.t. $z$ is 
solved at 
\begin{equation}\label{eq:lamz}
z = \alpha\cdot \rqs(x) \equiv \alpha \cdot [x^TA_1x, \dots, x^TA_mx]^T
\end{equation}
with $\alpha$ being a normalization factor for $\|z\|=1$ provided that 
$\rqs(x)\neq 0$. 
The formula of $z$ in~\eqref{eq:lamz} 
follows from $|z^T\rqs(x)|^2\leq \|\rqs(x)\|^2$ with 
equality holds if $z=\rqs(x)/\|\rqs(x)\|$.


Problem~\eqref{eq:tquartic} is a quartic maximization~\eqref{eq:fxqrt}
with real symmetric $A_i$ and real variables $x\in\R^n$.
By~\Cref{thm:vc}, the optimizer $x_*$ is an eigenvector of the mNEPv~\eqref{eq:hxqrt}. 
The corresponding eigenvalue
\begin{equation}\label{eq:ceig}
\lambda_* = x_*^TH(x_*)x_* = \sum_{k} \left(x_*^TA_kx_*\right)^2 = \mu_*^2.
\end{equation}
Note that any other eigenvalue $\lambda$ of~\eqref{eq:hxqrt} must satisfy
$\lambda \equiv x^TH(x)x = \sum_{k} \left(x^TA_kx\right)^2 \leq \lambda_*$,
due to~\eqref{eq:ceig} and maximization~\eqref{eq:tquartic}.

The best rank-one approximation is a fundamental problem in tensor analysis;
see~\cite{Lathauwer2000,Kofidis:2002,Zhang:2001}.
Third order partial-symmetric tensors are intensively studied~\cite{Carroll:1970,Li:2015,Qi:2018,Zhang:2012} and found 
in applications such as crystal structure~\cite{Chen:2021,Nye:1985}, 
where they are termed \emph{piezoelectric-type} tensors,
and modeling of social networks \Cref{eg:tensor}.  
It is known that tensor rank-one approximation problems are closely
related to tensor eigenvalue problems~\cite{Qi:2018},
such as the \emph{Z-eigenvalue}~\cite{Qi:2005} and
$\ell^2$-eigenvalue~\cite{Lim:2005} for general supersymmetric tensors 
and \emph{C-eigenvalue} for third order partial-symmetric 
tensors~\cite{Chen:2021}.  
Tensor eigenvalue problems provide first-order optimality 
conditions for the best rank-one approximation. 
But those eigenvalue problems are neither formulated nor studied
through the lens of the NEPv as presented in this paper. 
Particularly, for a third order partial-symmetric tensor,
its largest C-eigenvalue $\mu_*$ and the corresponding
C-eigenvectors $(x_*,z_*)$ form the best rank-one approximation
from~\eqref{eq:tensor}; see, e.g.,~\cite{Chen:2021}.
Whereas the tensor C-eigenvalue problem consists of two (coupled)
nonlinear equations in $(\mu, x, z)$, which are fundamentally different
from the mNEPv~\eqref{eq:hxqrt}.
How to efficiently solve those nonlinear equations for the C-eigenvalue 
is still largely open.



\subsection{Distance problem in dHDAE systems} \label{eg:dhdae}  
Consider the following 
dissipative Hamiltonian 
differential-algebraic equation (dHDAE):  
\begin{equation}\label{eq:structured}
J \frac{d^ju}{dt^j} 
= B_0 + B_1 \frac{du}{dt} + \cdots+B_\ell \frac{d^\ell u}{dt^\ell},
\end{equation}
where 
$u\colon \R \to\R^n$ is a state function,
$j$ is an integer between $0$ and $\ell$,
$J=-J^T$ is skew symmetric, and $B_i\succeq 0$ are symmetric
positive semi-definite for $i=0,\dots,\ell$. 
By convention, $\frac{d^0u}{dt^0} = u$.
The dHDAE~\eqref{eq:structured} arises  in energy based modeling 
of dynamical systems~\cite{Mehl:2021,Van:2014}.
An important special case is with $j=0$ and $\ell=1$, 
known as the linear time-invariant dHDAE system~\cite{Beattie:2018,Van:2014}.
Another one is the second-order dHDAE \eqref{eq:structured} with $j=1$ 
and $\ell=2$~\cite{Beattie:2018,Mehl:2021}.

To analyze the dynamical properties of a dHDAE system,
one needs to know whether the system is close to a singular one.
A dHDAE system~\eqref{eq:structured} is called \emph{singular} 
if $\det(P(\lambda)) \equiv 0$ for all $\lambda\in \C$, where
$P(\lambda)$ is the characteristic matrix polynomial defined by 
\begin{equation}\label{eq:llam}
P(\lambda) = -\lambda^j J 
	+ B_0 + \lambda B_1+\cdots+\lambda^\ell B_\ell.
\end{equation}
The distance of a dHDAE system to the closest singular 
dHDAE system is measured by the quantity $d_{\rm sing} (P(\lambda))$, which 
can be evaluated through the following optimization problem:
\begin{align} \label{eq:dsing}
d_{\rm sing} (P(\lambda))
 = \min_{x\in\R^n \atop \|x\|=1}
\left\{ 2\|Jx\|^2 + \sum_{i=0}^\ell \Big( 2\|(I-xx^T)B_ix\|^2 + 
(x^TB_ix)^2\Big) \right\}^{1/2},
\end{align}
see~\cite[Thm.16]{Mehl:2021}.
Let us show that the optimization~\eqref{eq:dsing} 
can be reformulated as the mNEPv~\eqref{eq:nepv}.  
First, by the skew-symmetry of $J$ and the symmetry of $B_i$, 
we can write~\eqref{eq:dsing} as 
\begin{align}
\Big( d_{\rm sing} (P(\lambda))\Big)^2  
& =\min_{x\in\R^n \atop \|x\|=1} 
 \left\{ 2\cdot x^T(J^TJ)x + 
\sum_{i=0}^\ell \left[ 2x^T(B_i^TB_i)x - (x^TB_ix)^2\right] \right\} \nonumber \\  
& = -2\cdot \max_{x\in\R^n \atop \|x\|=1}
\left\{ x^TA_1x + \frac{1}{2}\sum_{i=2}^{\ell+2} (x^TA_ix)^2 \right\}, \label{eq:dispoly}
\end{align}
where $A_1 \equiv J^2-\sum_{i=0}^\ell B_i^2$ 
and $A_{i} \equiv B_{i-2}$ for $i=2,\dots,\ell+2$.
Consequently,~\eqref{eq:dispoly} 
is of the form of the aMax~\eqref{eq:maxf}:
\begin{equation}\label{eq:fxdae}
\max_{x\in\R^n,\, \|x\|=1}
\Big\{ \objf(x) := x^TA_1x + \frac{1}{2}\sum_{i=2}^{\ell+2} \left(x^TA_ix\right)^2
\Big\},
\end{equation}
with $\phi_1(t)=t$ and $\phi_i(t)=t^2/2$ for $i=2,\dots,\ell+2$.
By the variational characterization in \Cref{thm:vc}, 
a local maximizer of \eqref{eq:fxdae} can be found by solving 
the following mNEPv of the form~\eqref{eq:nepv}:
\begin{equation}\label{eq:hxdae}
H(x)x=\lambda x
\quad\text{with}\quad
H(x) \equiv A_1+ \sum_{i=2}^{\ell+2} (x^TA_ix)\cdot A_i.
\end{equation}
where $h_1(t) = 1$ and $h_i(t) = t$ for $i=2,\dots,\ell + 2$ are non-decreasing 
and differentiable functions.

There are a couple of studies on estimating the
upper and lower bounds of the quantity
$d_{\rm sing}(P(\lambda))$~\cite{Mehl:2021,Prajapati:2022}. 
For linear systems, there is a recent work for estimating 
$d_{\rm sing} (P(\lambda))$ with a two-level minimization 
using ODE-based gradient flow~\cite{Guglielmi:2022}.
The mNEPv approach provides an alternative for estimating 
$d_{\rm sing} (P(\lambda))$ of dHDAE systems of an arbitrary order; 
see Examples~\ref{eg:dhdae2} and \ref{eg:dhdae3} 
in \Cref{sec:numex}.



\section{Numerical examples}  \label{sec:numex}
In this section, we present numerical examples of~\Cref{alg:scflf}
for solving the mNEPv~\eqref{eq:nepv} arising from the applications 
described in~\Cref{sec:nepvegs}.  The main purpose of the experiments 
is to illustrate the convergence behavior of the SCF
(\Cref{alg:scflf} with $\mbox{tol}_{\rm acc}=0$)
and the efficiency of accelerated SCF
(\Cref{alg:scflf} with $\mbox{tol}_{\rm acc}=0.1$).
The error tolerance for both algorithms are set to $\mbox{tol}=10^{-13}$.
All experiments are carried out in MATLAB 
and run on a Dell desktop with Intel i9-9900K CPU@3.6GHZ 
and 16GB core memory. 


\begin{example} \label{eg:rand}
{\rm  ~ 
Consider the computation of 
the numerical radius for a matrix $B \in \mathbb{C}^{n\times n}$.
As discussed in~\Cref{eg:quatic}, the related mNEPv is given by~\eqref{eq:hxqrt0}
and the variational characterization is by the
optimization~\eqref{eq:numrdmax0}
with Hermitian matrices $A_1=(B^H+B)/2$ and $A_2=(B^H-B)\cdot \imath/2$.
For numerical experiment, we consider the following matrix
\begin{equation}\label{eq:bmat}
    B = 
    \begin{bmatrix}
    0.6&   -0.2&   -1.9&   -0.3\\
   -0.1&   -0.3&   -1.3&   -1.2\\
   -2.0&   -1.6&   -2.1&    1.3\\
   -0.1&   -1.6&    1.5&   -0.1
    \end{bmatrix}
    + \imath
    \begin{bmatrix}
	0.6&    2.5&   -0.2&    2.5\\
    2.3&   -2.6&    0.4&    1.3\\
    0.0&    0.6&   -0.4&    1.2\\
    2.0&    1.4&    1.0&   -2.3
    \end{bmatrix}.
\end{equation}
The corresponding numerical range $W(A_1,A_2)$ 
is depicted in the left plot of~\Cref{fig:numrd:randnr} as the shaded region.
Following the discussion in~\Cref{sec:geo}, 
we sampled $100$ different starting vectors $x_0$
to run the SCF, where each $y_0 = \rqs(x_0)$ is a supporting point of
$W(A_1,A_2)$, depicted in~\Cref{fig:numrd:randnr} as dots on
the boundary of $W(A_1,A_2)$. 
Recalling the discussion on the implementation of~\Cref{alg:scflf}, 
such initial iterates $x_0$ can be obtained from the 
eigenvectors $x_{v}$ of the matrix $H_{v}$ 
for sampled directions $v\in\R^2$ (see~\Cref{lem:support} and~\Cref{sec:impl}).
Note that we can represent a unit direction $v\in\R^2$ by 
polar coordinates as $v=[\cos\theta, \sin\theta]^T$ with
$\theta\in[0,2\pi)$.
The initial vectors $x_0$ are set as
\begin{equation}\label{eq:ex1init}
	x_0 := x_v \quad \text{with $v=[\cos\theta, \sin\theta]^T$},
\end{equation} 
using $100$ equally distant $\theta$ between $0$ and $2\pi$.
The sampled $\rqs(x_0)$ are well distributed on the boundary of
$W(A_1,A_2)$, as shown in~\Cref{fig:numrd:randnr}.

\begin{figure}[t]
\centering\includegraphics[width=0.42\textwidth]{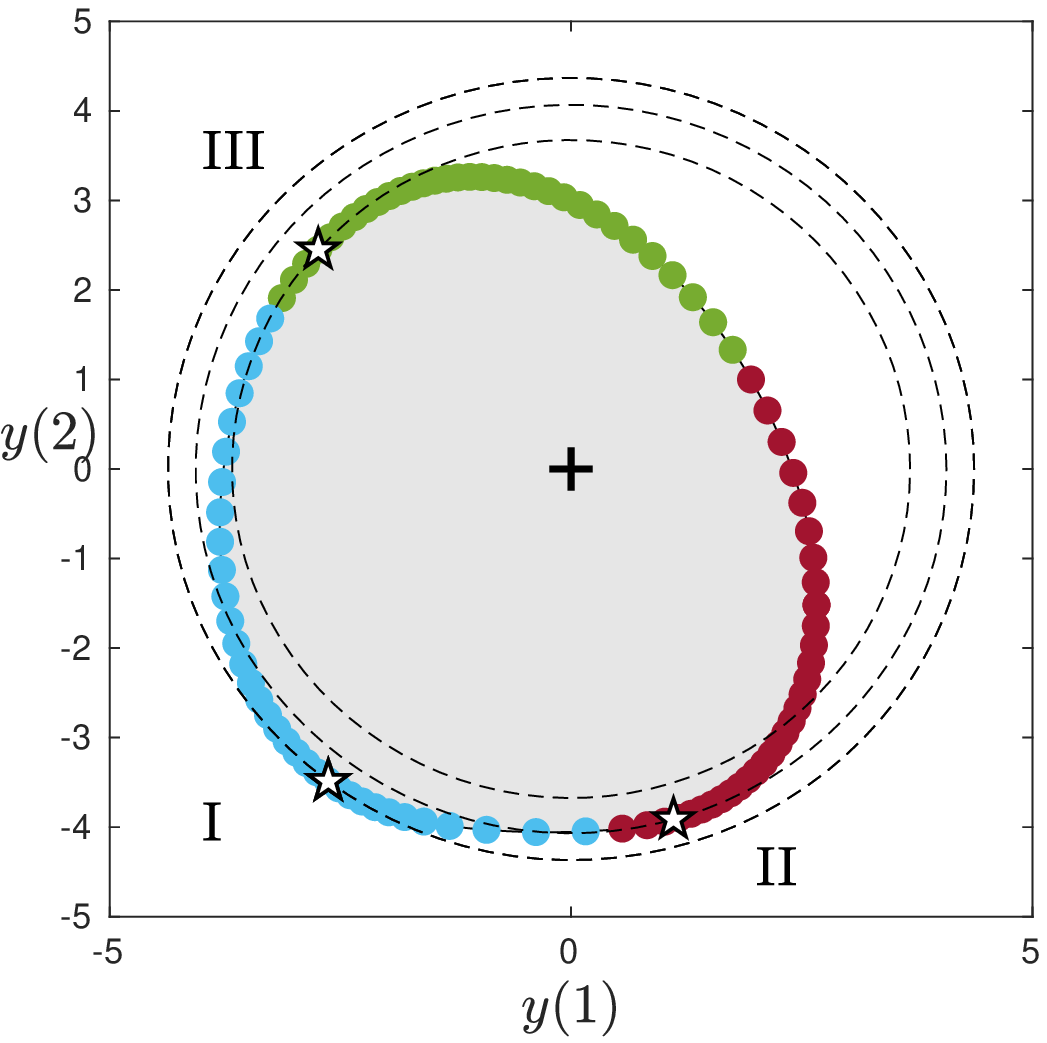}
\includegraphics[width=0.45\textwidth]{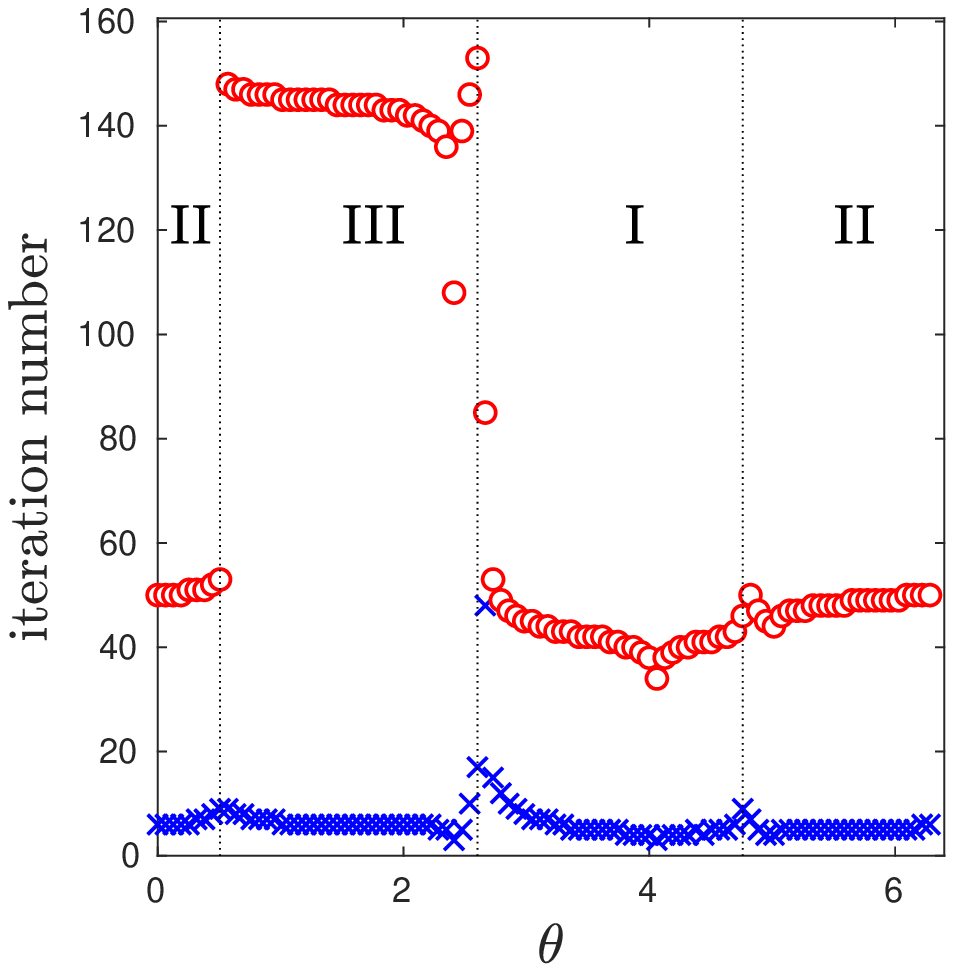}
\caption{
Left: Numerical range $W(A_1,A_2)$ of the matrix in~\eqref{eq:bmat},
where {\FiveStarOpen} represent the solution for the mNEPv
and $\bullet$ the starting $\rqs(x_0)$ of the SCF;
The $\bullet$ are colored according to the solution they have computed
(blue is for solution I, red for II, and green for III);
The dashed are contours of
$\phi(y) = 2^{-1}\|y\|_2^2 $; see~\eqref{eq:numrdmax0}.
Right: Number of iterations by the SCF
(marked `o') and accelerated SCF (marked  `$\times$')
for different starting $x_0$, parameterized by
$\theta\in[0,2\pi)$ as in~\eqref{eq:ex1init}.
}\label{fig:numrd:randnr}
\end{figure}

For $100$ runs of the SCF, three different solutions are found.
In~\Cref{fig:numrd:randnr}, they are labeled respectively with I, II, III, 
in descending order of their objective values of~\eqref{eq:numrdmax0}.
The initial $\rqs(x_0)$ on the boundary of $W(A_1,A_2)$ 
are colored the same if SCF will converge to the same solution,
which, hence, reveals the region of convergence for SCF.
The numbers of SCF iterations with each $x_0$ are reported in the right
plot of~\Cref{fig:numrd:randnr}. 
We can see that the iterations are determined by the 
eigenvector SCF computed, and they stay almost flat for that eigenvector.
For the SCF, the iteration numbers vary for different solution.
Whereas for the accelerated SCF, they are almost independent of the
choice of the initial $x_0$, with only a moderate increase on 
the boundary of two convergence regions.

The left plot of~\Cref{fig:numrd:rand}
depicts the convergence history of the objective function $\objf(x_k)$
for four different starting vectors $x_0$, 
corresponding to the equally distant 
$\theta \in \{0,\,\pi/2,\, \pi,\, 3\pi/2\}$
from~\Cref{fig:numrd:randnr}.
As expected the SCF demonstrates monotonic convergence.
The right plot in~\Cref{fig:numrd:rand} shows 
the relative residual norms of $(\lambda_k, x_k)$ as defined in~\eqref{eq:resd}.
We can see that the SCF quickly 
enters the region of linear convergence in all cases (in about 3 iterations).
The acceleration takes full advantage of the rapid
initial convergence and speeds up the SCF significantly. 
We note that in this example the matrices $A_1$ and $A_2$ are complex
Hermitian. The inverse iteration~\eqref{eq:syminverse} with Rayleigh shift
$\sigma_k$ is not guaranteed quadratically convergent. 

\begin{figure}[t]
\centering
\includegraphics[width=0.45\textwidth]{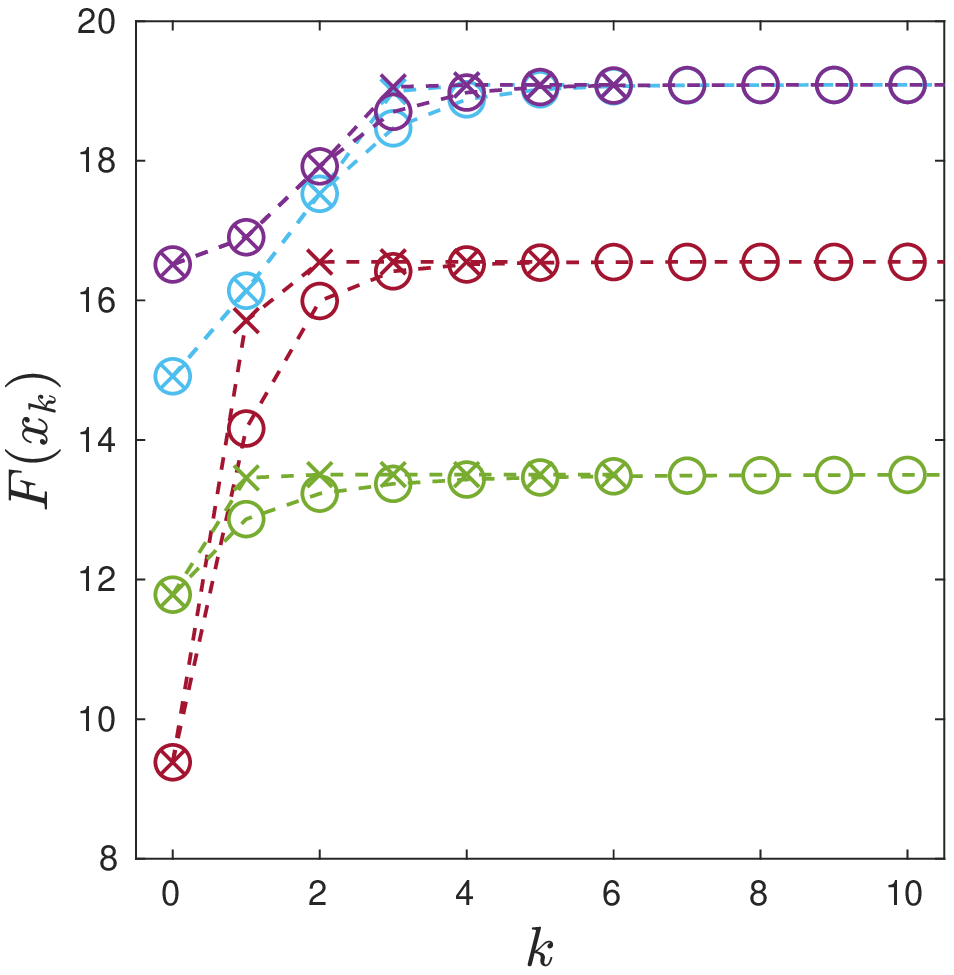}
\includegraphics[width=0.45\textwidth]{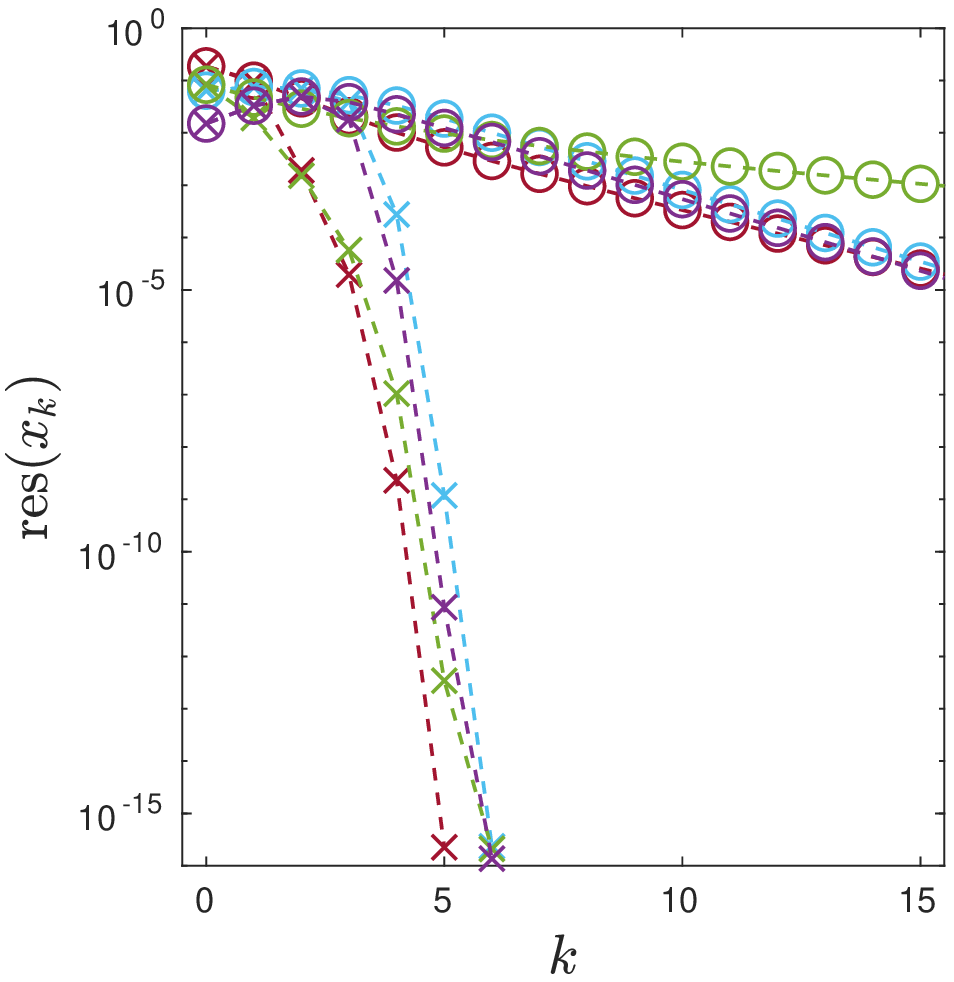}
\caption{
\Cref{eg:rand}.
Left: Convergence history of $\objf(x_k)$ by the SCF
(marked `o') and accelerated SCF (marked  `$\times$'),
where each colored curve is a run with
a particular $x_0$ from $4$ different starting vectors.
Right: Relative residual norms~\eqref{eq:resd}. 
}\label{fig:numrd:rand}
\end{figure}


} \end{example}


\begin{example} \label{eg:dhdae2} 
{\rm 
We consider the mNEPv~\eqref{eq:hxdae}
arising in the distance problem of dHDAE systems described in~\Cref{eg:dhdae}.
The characteristic polynomial of a linear dHDAE system is given by
\begin{equation}\label{eq:plep}
P(\lambda) :=-J +R + \lambda E,
\end{equation}
where $J=-J^T$ is a skew symmetric, and $E$ and $R$ 
are symmetric positive definite matrices.
As discussed in~\Cref{eg:dhdae}, the computation of 
distance to singularity $d_{\rm sing} (P(\lambda))$
leads to the optimization~\eqref{eq:dispoly} 
and the associated mNEPv~\eqref{eq:hxdae}, 
where
\begin{equation}\label{eq:dhdaelin}
\objf(x)= x^TA_1x + \frac{1}{2}\sum_{i=2}^3(x^TA_ix)^2
\quad \mbox{and} \quad
H(x) = A_1+ \sum_{i=2}^3(x^TA_ix)\cdot A_i,
\end{equation}
and $A_1 = J^2-E^2-R^2$, $A_2 = E$ and $A_3 = R$.

For numerical experiments, 
the matrices $\{J,R,E\}$ of order 30 are 
generated randomly.\footnote{For the positive definite $E$ and $R$, we use:
	\texttt{X=randn(n); X = orth(X); X = X*diag(rand(n,1)+1.6E-6)*X'}.
		For the skew symmetric $J$, we use:
	\texttt{X=randn(n); X = X-X'; X = X/norm(X)}.
	}
Similar to~\Cref{eg:rand}, the initial vectors $x_0$ of the SCF are computed 
from supporting points of the joint numerical range 
$W(A_1,A_2,A_3)\subset\R^3$ along several sampled direction $v\in\R^3$.
Here, recall that a unit $v\in\R^3$ 
can be represented by spherical coordinates as 
\begin{equation}\label{eq:dhdaeinit}
	v=[\sin\eta \cos\theta,\, \sin\eta\sin\theta,\, \cos\eta]^T
	\quad\text{with $\eta\in[0,2)$ and $\theta\in[0,2\pi)$}.
\end{equation}
We have therefore constructed an equispaced grid of $20$-by-$40$ points of
$(\eta,\theta)\in[0,\pi]\times [0,2\pi]$, 
yielding $800$ supporting points of $W(A_1,A_2,A_3)$.
They are depicted in the left plot of~\Cref{fig:dHDAE:nr},
together with the approximate numerical range 
$W(A_1,A_2,A_3)$ they generate.\footnote{
	Plot generated by MATLAB functions
	\texttt{trisurf} and~\texttt{boundary}
	using the $800$ sampled supporting points.
}

\begin{figure}[t]
\centering 
\includegraphics[width=0.54\textwidth]{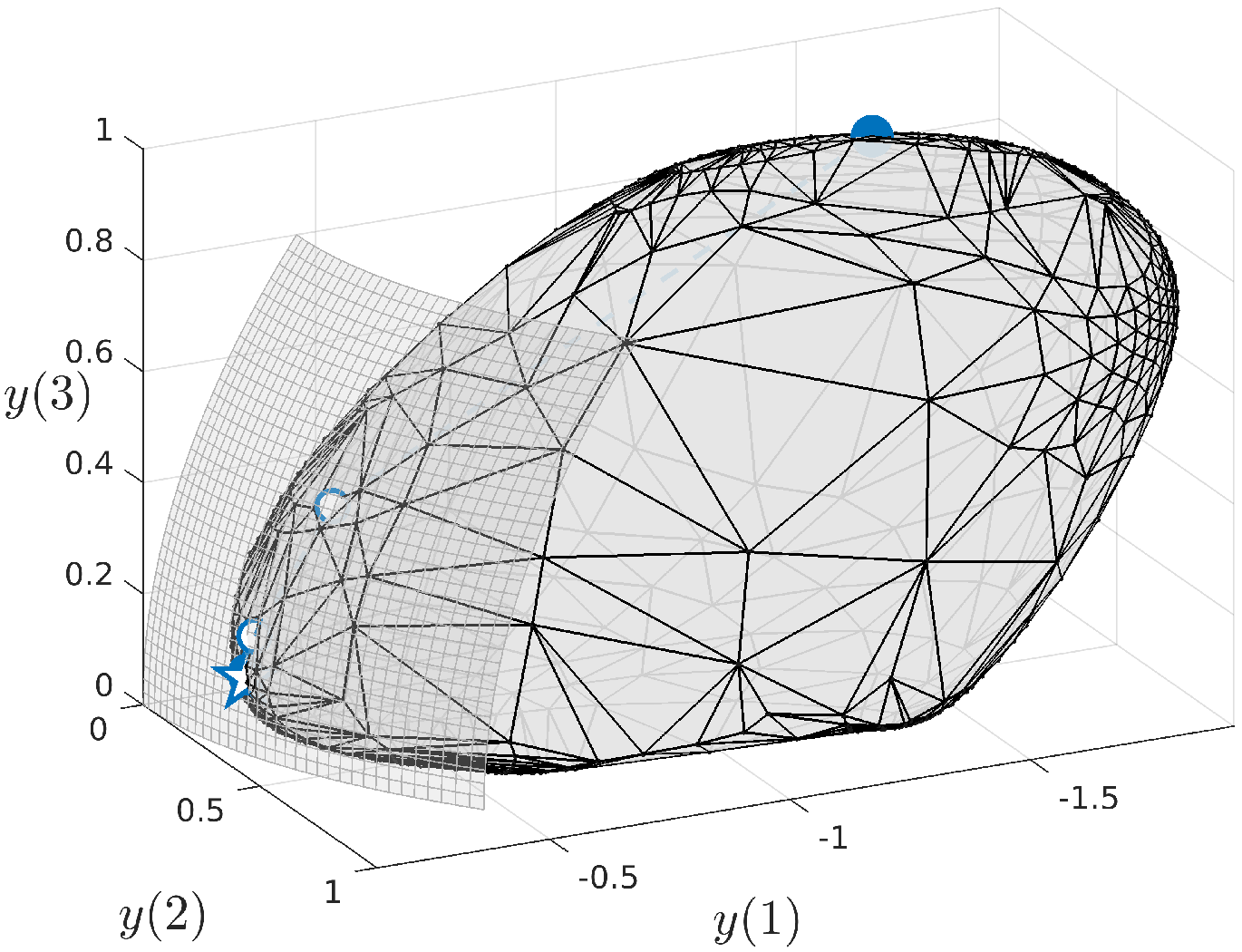}
\includegraphics[width=0.44\textwidth]{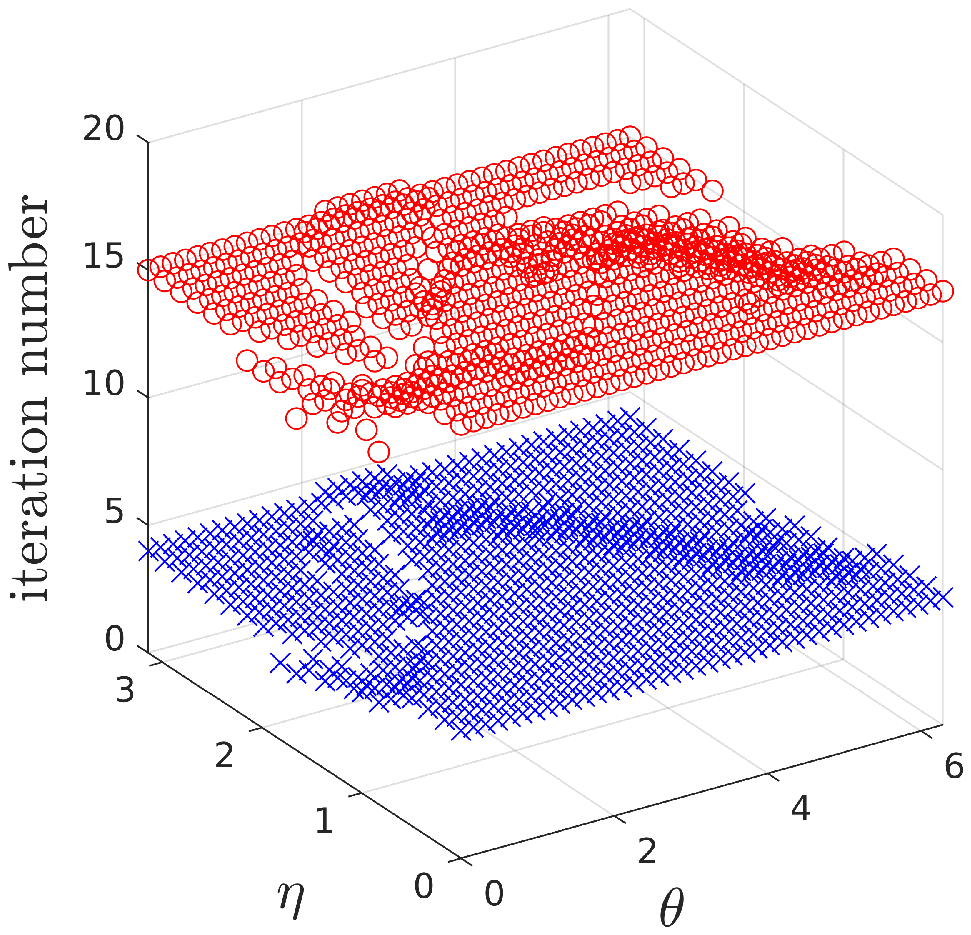}
\caption{
Left: Approximate numerical range $W(A_1,A_2,A_3)$, based on $800$ 
sample supporting points on the boundary (nodes of the mesh);
The $\text{\FiveStarOpen}$ represents the solution for the mNEPv, 
$\bullet$ the starting $\rqs(x_0)$, and `$\circ$' the first few
supporting points $\rqs(x_k)$ by SCF;
The smaller mesh that crosses $\text{\FiveStarOpen}$ is part of the level-surface 
$\phi(y) = \phi(y_*)$ for $\phi(y)=y(1)+(y(2)^2+y(3)^2)/2$
at the solution $y_*=\rqs(\widehat x_*)$.
Right:
Number of iterations by the SCF
(marked `o') and the accelerated SCF (marked  `$\times$')
for different starting $x_0$, parameterized by 
$\theta\in[0,2\pi)$ and $\eta\in[0, \pi)$ as in~\eqref{eq:dhdaeinit}.
}\label{fig:dHDAE:nr}
\end{figure}

From all $800$ initial $x_0$, the SCF converge to a same solution, as
marked in the left plot of~\Cref{fig:dHDAE:nr}.
This solution appears to be the global optimizer of the
optimization~\eqref{eq:dispoly}, as visually verified by the
level-surface of the objective function $\phi(y)$ for the
corresponding optimization over the joint numerical range~\eqref{eq:maxphi}.
From the numbers of iterations reported in the right plot
of~\Cref{fig:dHDAE:nr}, we can see that both 
SCF and accelerated SCF converge rapidly to the solution.
The iteration numbers are not sensitive to the choice of initial
vectors.


\Cref{fig:dHDAE:lin} depicts the convergence history of $\objf(x_k)$ 
and the relative residual norms by the 
SCF from six different starting vectors
$x_0$ (sampled supporting points of $W(A_1,A_2,A_2)$ along the three coordinate axes).
We observe that the SCF with different starting vector
converges monotonically to the same solution. 
The accelerated SCF greatly reduces the number of iterations and 
shows the quadratic convergence rate.

\begin{figure}[t]
\centering 
\includegraphics[width=0.45\textwidth]{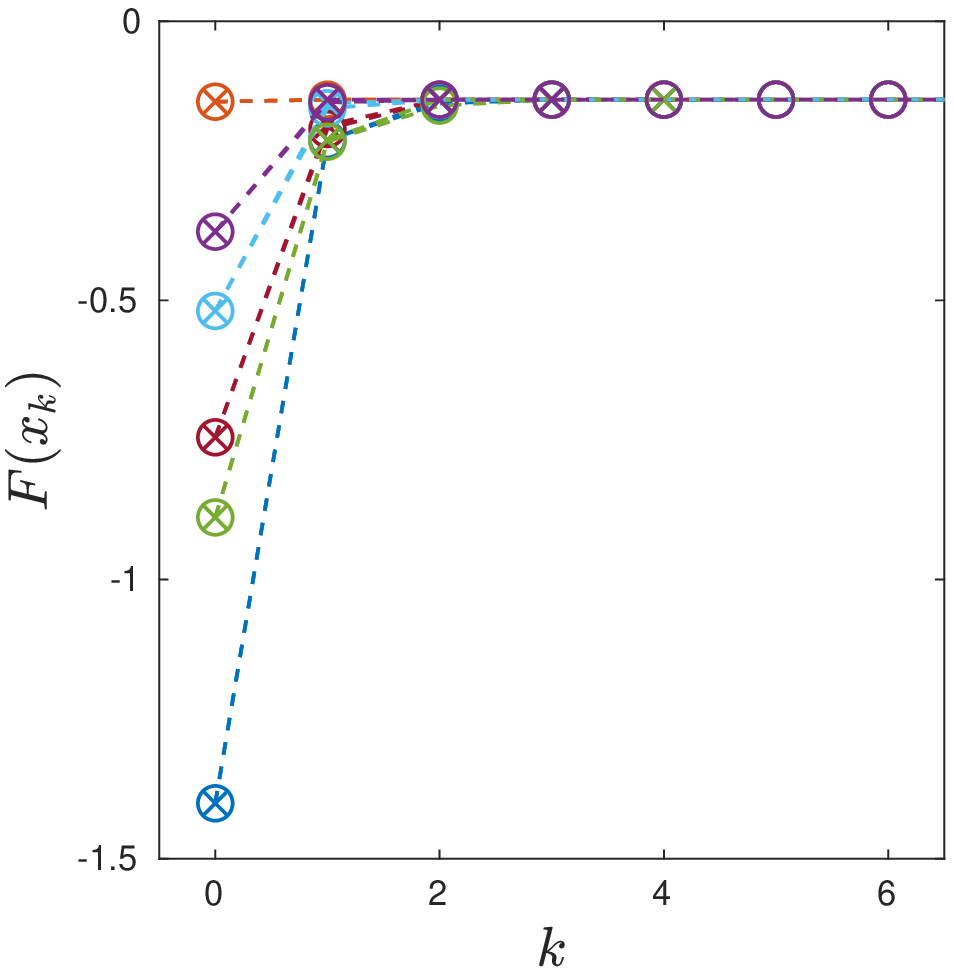}
\includegraphics[width=0.45\textwidth]{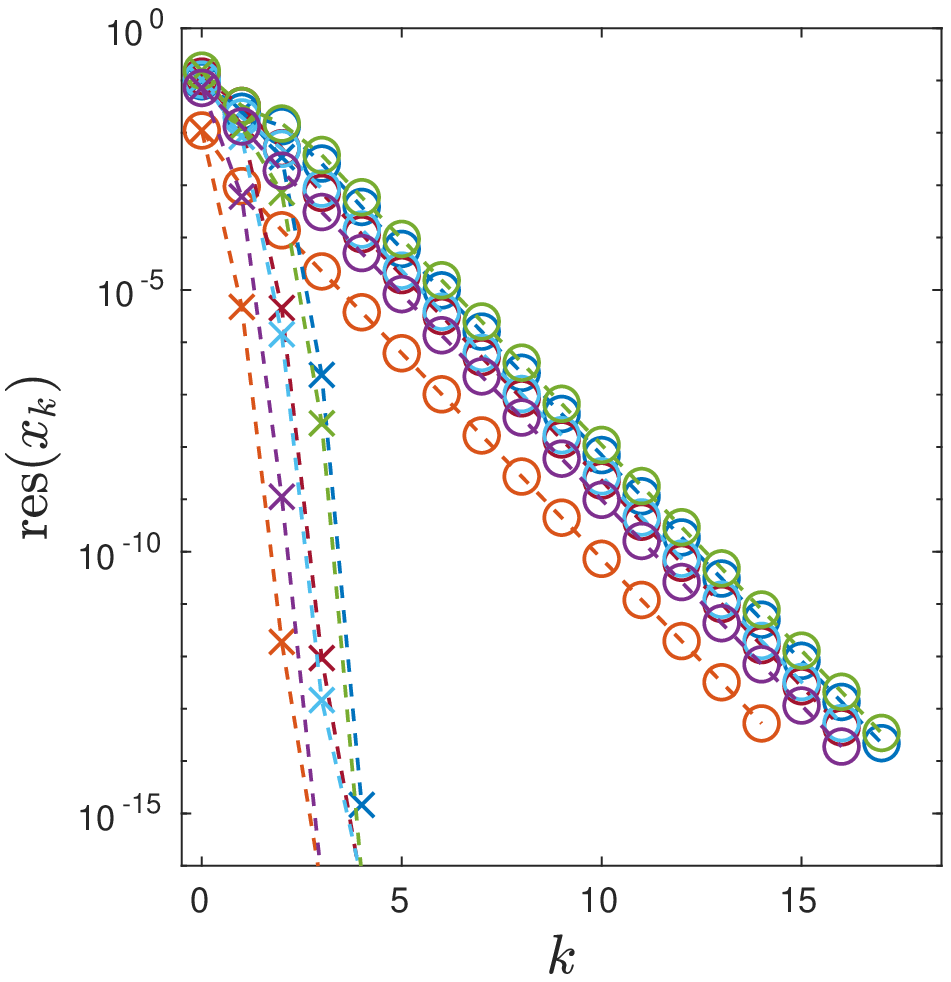}
\caption{
Left: Convergence history of $\objf(x_k)$ by the SCF
(marked `o') and accelerated SCF (marked  `$\times$'),
where each colored curve is a run with
a particular $x_0$ from $6$ different starting vectors.
Right: Relative residual norms~\eqref{eq:resd}. 
}\label{fig:dHDAE:lin}
\end{figure}

Recall that a computed $\widehat x_*$ may not be a global maximizer
of the aMax~\eqref{eq:dispoly}.
But we have at least an upper bound 
of the distance: 
\begin{equation}\label{eq:newBD}
d_{\rm sing} (P(\lambda))
\equiv \big(-2\cdot \max_{\|x\|=1} \objf(x)\big)^{1/2} \leq
\big(-2\cdot \objf(\widehat x_*)\big)^{1/2}.
\end{equation}
If the initial vector $x_0$ of the SCF is especially set to be
the eigenvector corresponding to the largest eigenvalue of $A_1$, 
then we have
\begin{equation}\label{eq:MehlBD}
         \big(-2\cdot \objf(\widehat x_*)\big)^{1/2}
         \leq
         \big(-2\cdot \objf(x_0)\big)^{1/2}
         \leq \delta_M := \big(-2\cdot \lambda_{\max}(A_1)\big)^{1/2},
\end{equation}
where the first inequality is by the monotonicity of the SCF
(see~\Cref{thm:mono})
and the second inequality is by the definition of $\objf(x)$~\eqref{eq:dhdaelin}
(recall that $R$ and $E$ are positive definite).
The quantity $\delta_M$ was introduced in~\cite[Thms.13 and 16]{Mehl:2021}
and used as an estimation of the quantity $d_{\rm sing} (P(\lambda))$.
It follows from the inequalities~\eqref{eq:newBD} and~\eqref{eq:MehlBD} 
that the SCF always produces a sharper upper bound 
of $d_{\rm sing} (P(\lambda))$. 
For example, in a numerical example, the SCF provides
an estimation $\big(-2\cdot \objf(\widehat x_*)\big)^{1/2} \approx 0.5989$. 
In contrast, $\delta_M \approx 0.6923$.

We note that the quantity $\delta_M$ has been revisited
in a recent work~\cite{Prajapati:2022}, where a computable upper bound was
proposed.  The latter, however, involves a more complicated optimization
of sum of generalized Rayleigh quotients 
and does not ensure a better estimation than $\delta_M$;
see~\cite[Thm.~3.7 and Example~3]{Prajapati:2022}.
In another related work~\cite{Guglielmi:2022}, 
the authors considered an
approach to estimate $d_{\rm sing}(P(\lambda))$,
based on the observation that $d_{\rm sing} (P(\lambda))$ is the smallest root of
a monotonically decreasing function $w$. A root finding method such as
the bisection can be applied.
The difficulty there lies in the evaluation of the function $w$. 
For a given $\epsilon$, evaluating $w(\epsilon)$ can be very 
expensive as it requires the solution of an optimization by 
a gradient flow method, which involves repeated solution
of Hermitian eigenvalue problems of size $n$.
} 
\end{example}

\begin{example} \label{eg:dhdae3} 
{\rm 
In this example, we consider a 
quadratic dHDAE system 
with a characteristic polynomial 
$$
P(\lambda) :=-\lambda G + K + \lambda D + \lambda^2 M,
$$
where $G=-G^T$ is skew symmetric, and $M$, $D$ and $K$ are symmetric
positive definite.
By~\Cref{eg:dhdae}, the computation of distance to singularity 
$d_{\rm sing} (P(\lambda))$
leads to the optimization~\eqref{eq:dispoly} and 
the mNEPv~\eqref{eq:hxdae}, where
\[
\objf(x)= x^TA_1x + \frac{1}{2}\sum_{i=2}^4(x^TA_ix)^2
\quad \mbox{and} \quad 
H(x) = A_1+ \sum_{i=2}^4(x^TA_ix)\cdot A_i,
\]
where 
$A_1 = G^2-M^2-D^2-K^2$, $A_2 = M$, $A_3 = D$, and $A_4 = K$.  

For numerical experiments, we consider a lumped-parameter 
mass-spring-damper system, 
$M\ddot u + D\dot u + Ku = f$; see, e.g.,~\cite{Veselic:2011},
with $n$ point-masses and $n$ spring-damper pairs.
The matrices $D$ and $K$ are interchangeable with
$DK=KD$ and are simultaneously diagonalizable.
We pick a random skew symmetric $G$ as in~\Cref{eg:quatic} 
to simulate the gyroscopic effect. 
The sizes $n$ of the matrices are set 
ranging from $500$ to $3000$.
For each set of testing matrices, we apply the SCF 
with $100$ different starting vectors $x_0$.
Again, those $x_0$ are computed from supporting points of
the joint-numerical range $W(\mathcal  A)\subset\R^4$
along $100$ randomly sampled directions $v\in\R^4$.

Similar to the linear system in~\Cref{eg:dhdae2}, 
the SCF always converge to a same solution from all $100$ 
different starting vectors.
The convergence history of the SCF and the accelerated SCF 
for a case of $n=1000$, with $8$ randomly selected starting vectors are 
depicted in~\Cref{fig:dHDAE:quad}.
It shows a same convergence behavior of the SCF and accelerated SCF 
as in the previous example.

\begin{figure}[t]
\centering 
\includegraphics[width=0.45\textwidth]{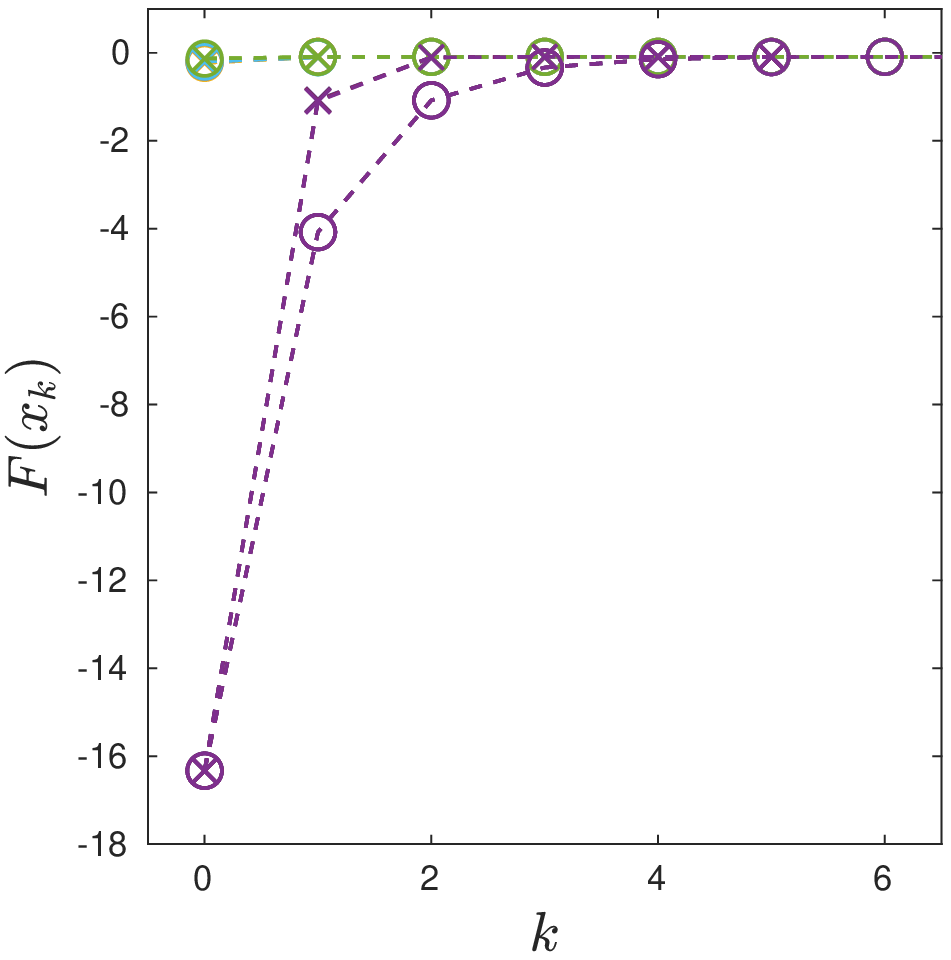}
\includegraphics[width=0.45\textwidth]{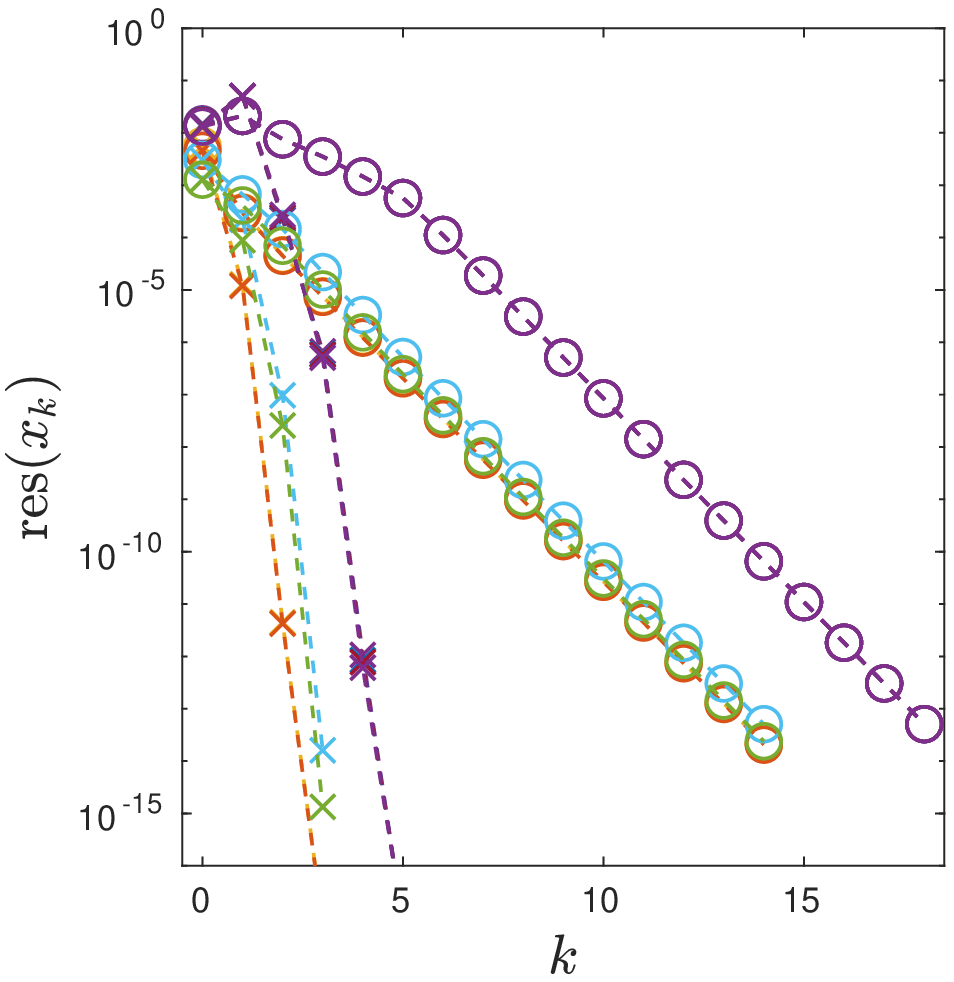}
\caption{
	Left: Convergence history of $\objf(x_k)$ by the SCF
	(marked `o') and accelerated SCF (marked  `$\times$'),
	where each colored curve is a run with
	a particular $x_0$ from $8$ 
	different starting vectors (lines overlapped).
	Right: Relative residual norms~\eqref{eq:resd}. 
}\label{fig:dHDAE:quad}
\end{figure}

\Cref{tab:dhdae3} summarizes the iteration number and computation time for
the algorithms from all testing cases.
We can see that the performance of both SCF and accelerated 
SCF are not much affected by the choice of initial vectors. 
Both algorithms converge rapidly,
and the accelerated SCF speed up to a factor between $2.5$ to $6.2$.

\begin{table}[tbhp]
\caption{
Number of iterations and computation time (in seconds) for various
problem sizes $n$.
Reported are average results from $100$ runs with different starting
vectors, with the largest deviations also marked.
}\label{tab:dhdae3}
\centering
\begin{tabular}{c|c|c|c|c}
\toprule
$n$ & algorithms & iterations & timing & speedup\\ 
\midrule 
\multirow{2}{*}{$500$}
& SCF 	& $17.0$ ($\pm 4.0$)& $1.18$ ($\pm 0.34$) & -- \\ 
& accel. SCF & ~$5.3$ ($\pm 1.3$)& $0.34$ ($\pm 0.14$) & $2.9$ to $4.5$\\ 
\midrule 
\multirow{2}{*}{$1000$}
& SCF & $21.3$ ($\pm 3.3$)& $6.54$ ($\pm 1.15$) & -- \\ 
& accel. SCF	 & ~$4.7$ ($\pm 1.3$)& $1.32$ ($\pm 0.48$) & $3.9$ to $6.2$\\ 
\midrule 
\multirow{2}{*}{$2000$}
& SCF & $17.0$ ($\pm 4.0$)& $4.91$ ($\pm 1.58$) & -- \\ 
& accel. SCF	 & ~$4.8$ ($\pm 1.8$)& $1.52$ ($\pm 0.73$) & $2.5$ to $4.9$\\ 
\midrule 
\multirow{2}{*}{$3000$}
& SCF & $16.9$ ($\pm 3.9$)& $20.51$ ($\pm 5.33$) & -- \\ 
& accel. SCF	 & ~$5.2$ ($\pm 1.2$)& $6.39$ ($\pm 1.93$)  & $2.7$ to $4.0$\\ 
\bottomrule
\end{tabular}
\end{table}
} \end{example}

 
\begin{example} \label{eg:tensor} 
{\rm ~
As discussed in~\Cref{sec:tensor},
the problem of best rank-one approximation for
a partial-symmetric tensor $T\in\R^{n\times n\times m}$ 
leads to a quartic optimization~\eqref{eq:tquartic}
and the corresponding mNEPv~\eqref{eq:hxqrt},
where the coefficient matrices
are $A_i:=T(:,:,i)\in\R^{n\times n}$ for $i=1,\dots,m$.

For non-negative tensors, the objective function $\objf(x)$
of~\eqref{eq:tquartic} 
satisfies $\objf(|x|) \geq \objf(x)$, where $|\cdot|$ denotes 
componentwise absolute value.
Therefore, it is advisable to start the SCF~\eqref{eq:scf}
with a non-negative initial $x_0$.
Note that if $x_k\geq 0$ then $H(x_k)\geq 0$.
By the Perron-Frobenius theorem~\cite{Horn:2012}, 
the eigenvector $x_{k+1}$ corresponding to the largest eigenvalue 
of $H(x_k)$ is also non-negative.
This implies that the subsequent iterates $x_k$ by 
the SCF will remain non-negative.

We note that for a non-negative tensor $T$ and 
a non-negative initial $x_0$, the SCF~\eqref{eq:scf} 
is indeed equivalent to the Alternating Least Squares (ALS) algorithm 
for finding the best rank-one approximation~\eqref{eq:tensor}.
Recall that in \Cref{sec:tensor}, 
the best rank-one approximation~\eqref{eq:tensor} is turned into 
the maximization problem: 
\begin{equation}\label{eq:alsobj}
    \max_{\|x\|=1,\, \|z\|=1}
    \left(z^T\cdot \rqs(x) \right)^2,
\end{equation}
where $\rqs(x) = [x^T A_1 x, \ldots, x^T A_m x]^T$.
Maximizing alternatively with respect to $z$ and $x$
leads to the alternating iteration: 
\begin{equation}\label{eq:als}
\left\{\begin{aligned}
    z_{k+1} & = \argmax_{\|z\|=1}
    \left(z^T\cdot \rqs(x_k) \right)^2
    = \alpha_k\cdot \rqs(x_k), \\
    x_{k+1} & =
    \argmax_{\|x\|=1} \left(z_{k+1}^T\cdot \rqs(x) \right)^2
    =
    \argmax_{\|x\|=1} \left(x^T\cdot H(x_k)\cdot x \right)^2,
\end{aligned}\right.
\end{equation}
for $k=1,2,3,\dots$, where $\alpha_k >0 $ is a normalization factor for $z_{k+1}$.
Recall that $H(x_k)\geq 0 $ if $x_k\geq 0$.
The maximizer $x_{k+1}$ of~\eqref{eq:als} 
is the eigenvector corresponding to the largest eigenvalue of $H(x_k)$,
i.e., $H(x_k) x_{k+1} = \lambda_{\max} \cdot x_{k+1}$, due to the
Perron-Frobenius theorem.
Therefore~\eqref{eq:als} coincides with the SCF.
The ALS algorithms are commonly used for low-rank approximations in 
tensor computations~\cite{Kolda:2009}.
The scheme~\eqref{eq:als} accounts for the partial symmetry of 
the rank-one tensor $\lambda\cdot x\otimes x\otimes z$ such that
only two vectors $x$ and $z$ are alternated.


In \Cref{alg:scflf},
we use MATLAB~\texttt{eigs} for solving the eigenvalue computation
and~\texttt{minres} for solving the linear system 
in the acceleration~\eqref{eq:syminverse}.
We use an adaptive error tolerance 
$\mbox{Tol}=\min\{10^{-3}, \mbox{res}(x_k)^2\}$ 
for each call of~\texttt{eigs} and~\texttt{minres}.

For numerical experiments, we use the following thrid order 
partial-symmetric tensors.
\begin{enumerate}
\item The {\em New Orleans tensor}~\footnote{
        From~\cite{Viswanath:2009}, available at 
            \url{http://socialnetworks.mpi-sws.org/data-wosn2009.html}.},
        created from the Facebook New Orleans network.
        The original data contains a list of all of the user-to-user
        links (undirected) and a timestamp for the establishment of the link. 
	The links are collected on a monthly (30 day) basis for $20$ months,
        with each month corresponding to a slice of $T(:,:,i)$. 
        The size of the resulting tensor $T$ is $63891\times 63891\times 20$
        with $477778$ nonzeros.

    \item 
        The {\em Princeton tensor}~\footnote{
        From~\cite{Traud:2012}, available at 
        \url{https://archive.org/details/oxford-2005-facebook-matrix.}},
        created from a Facebook `friendship' network at 
        Princeton, following the setting up in~\cite{Elden:2020}.
        The element $T(i,j,k)=1$ 
        if students $i$ and $j$ are friends and one of them has 
		a status flag $k$. 
        The size of the resulting tensor $T$ is 
        $6593\times 6593\times 6$ with $70248$ nonzeros. 

    \item
        The {\em Reuters tensor}~\footnote{
        From~\cite{Batagelj:2003}, available at
        \url{http://vlado.fmf.uni-lj.si/pub/networks/data/CRA/terror.htm}.},
        created from a news network
        based on all stories released by the news agency Reuters,
        concerning the September 11 attack
        during the 66 consecutive days beginning at September 11, 2001. 
        The vertices of the network are words,
        and there is an edge between two words if 
        they appear in the same sentence, 
        with the weight of an edge being the frequency. 
        The size of the tensor $T$ is $13332\times 13332\times 66$
        with $486894$ nonzeros.

\end{enumerate}
All three tensors are non-negative and sparse 
(density $\approx 10^{-5}$),
so are the corresponding coefficient matrices 
$A_i = T(:,:,i)$ for $i=1,\dots,m$.
We use $100$ randomly generated
and non-negative starting vectors $x_0$ to run the SCF
(using \texttt{x0=abs(randn(n,1))}).
The convergence history is reported in~\Cref{fig:tensor}.
We observe that from different starting vectors,
~\Cref{alg:scflf} always converge to the 
same solution and the convergence rate appears not affected 
by the choice of starting vectors.  
Also, the accelerated SCF significantly reduces
the number of the SCF iterations and has a quadratic convergence rate.

The computed optimal values of the objective function 
and timing of the ALS (i.e., the SCF) and
the accelerated SCF are reported in~\Cref{tab:tensor}. 
The accelerated SCF speeds up the ALS by a factor of 2.
This demonstrates one of the benefits of the NEPv reformulation
for allowing the development of effective acceleration scheme of the ALS.

\begin{figure}[t]
\centering 
\includegraphics[width=0.3\textwidth]{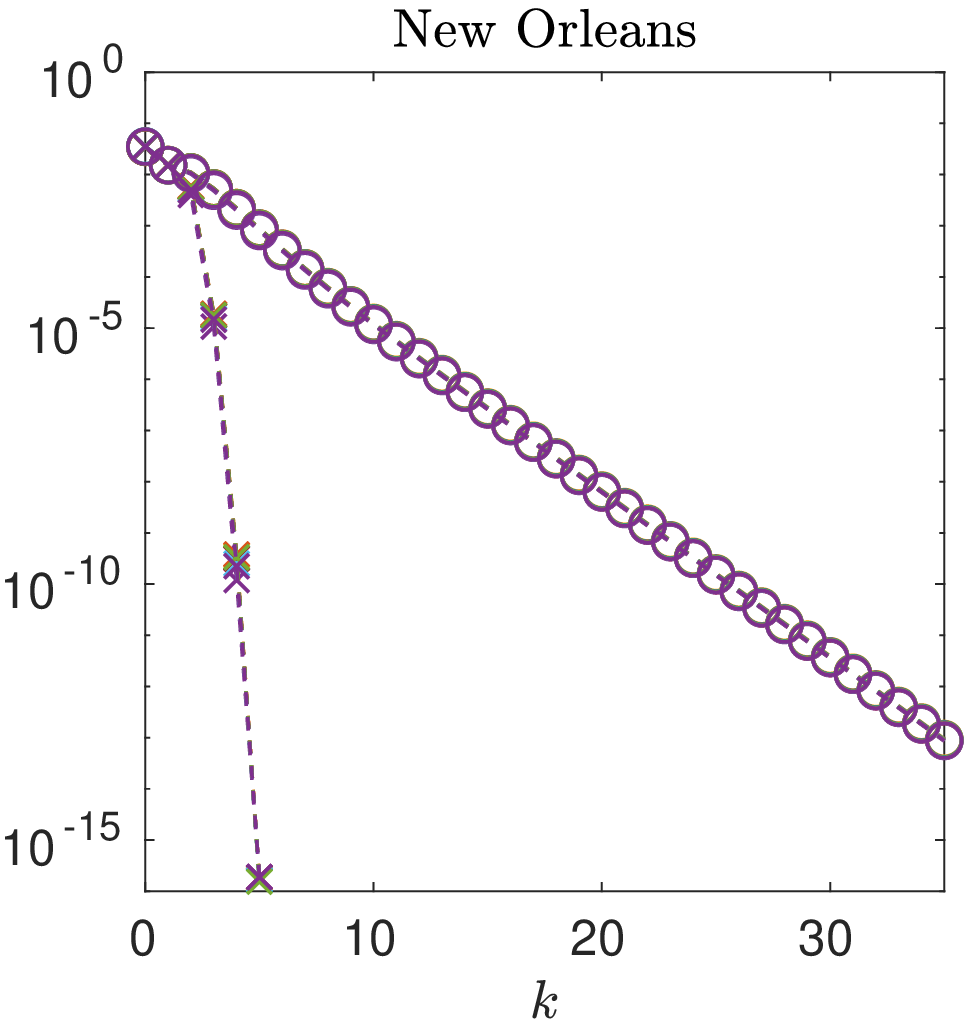}
\includegraphics[width=0.3\textwidth]{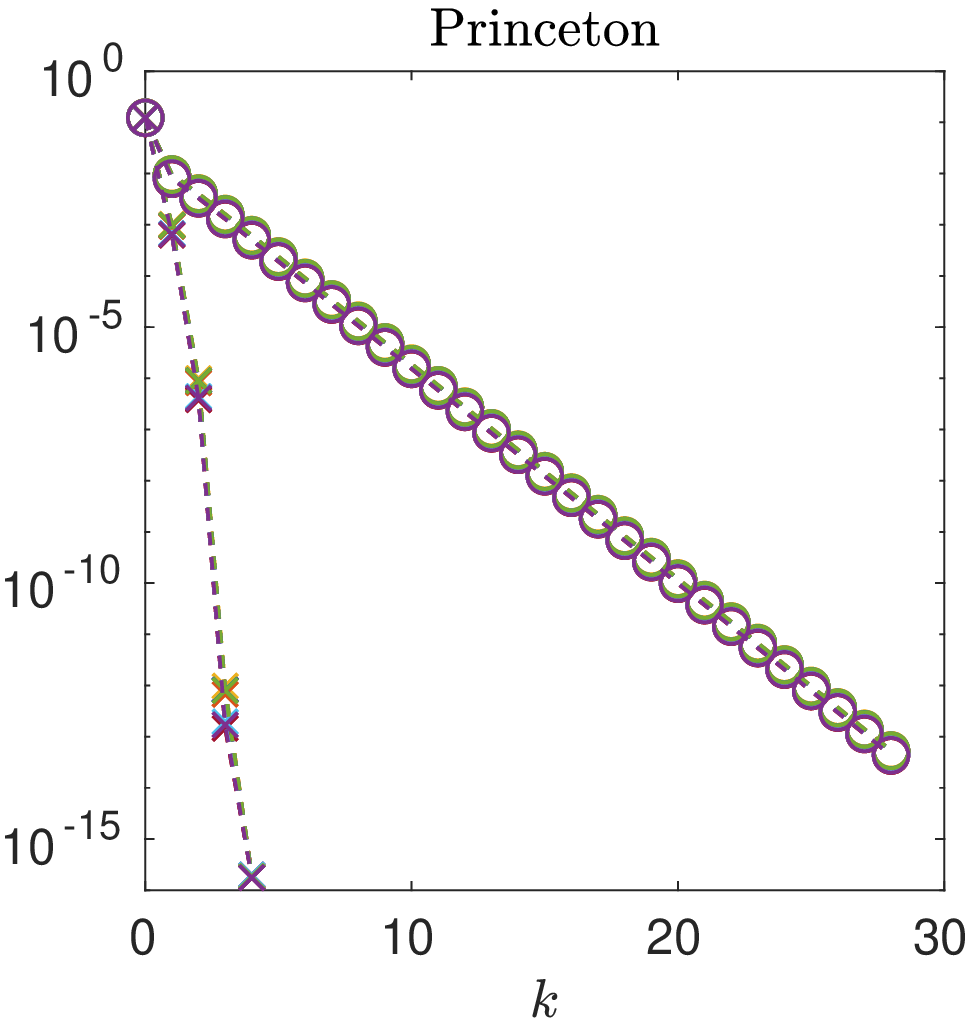}
\includegraphics[width=0.3\textwidth]{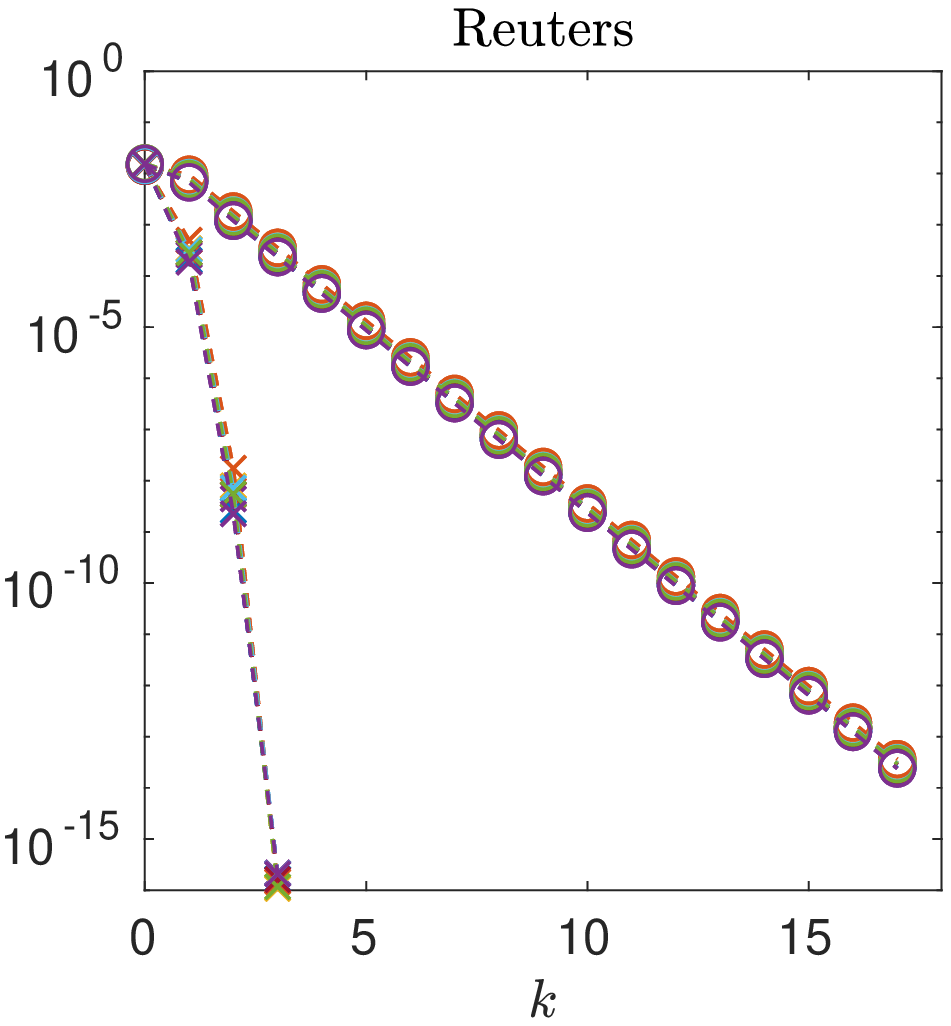}
\caption{
Convergence history of relative residual norms $\mbox{res}(x_k)$~\eqref{eq:resd}
by the SCF (`$\circ$') and accelerated SCF (`$\times$').
Each colored curve represents a run with a different starting vector
from $100$ randomly generated $x_0\geq 0$ (lines overlapped).
}\label{fig:tensor}
\end{figure}

\begin{table}[t]
\caption{
Objective function value $\objf(x_*)$ of the computed eigenvectors $x_*$,
total iteration number, and computation time (in seconds).
Reported are average results from $100$ runs with different starting
vectors, with the largest deviations also marked.
}\label{tab:tensor}
\begin{tabular}{c|c|c|c|c}
\toprule
tensor & algorithms & $\objf(x_*)$ & iterations & timing\\ 
\midrule 
\multirow{2}{*}{New Orleans}
& ALS & $257.1509714238355$ ($\pm 3\cdot 10^{-13}$) & $36$ & $1.69$ ($\pm 2\cdot 10^{-1}$)  \\ 
& accel. SCF &  $257.1509714238355$ ($\pm 3\cdot 10^{-13}$) & 6 & $0.75$ ($\pm 1\cdot 10^{-1}$)\\
\midrule 
\multirow{2}{*}{Princeton}
& ALS & $26441.19404229125$ ($\pm 3\cdot 10^{-11}$) & $30$ & $0.42$ ($\pm 4\cdot 10^{-2}$)  \\ 
& accel. SCF & $26441.19404229123$ ($\pm 2\cdot 10^{-12}$) & $5$ & $0.18$ ($\pm 5\cdot 10^{-2}$)  \\ 
\midrule 
\multirow{2}{*}{Reuters}
& ALS & $118777.1084768529$ ($\pm 1\cdot 10^{-10}$) & $18$ & $0.65$ ($\pm 4\cdot 10^{-2}$)  \\ 
& accel. SCF& $118777.1084768529$ ($\pm 1\cdot 10^{-10}$) & $4$ & $0.32$ ($\pm 2\cdot 10^{-2}$)  \\ 
\bottomrule
\end{tabular}
\end{table} 

}
\end{example}



\section{Concluding remarks}\label{sec:conclusion}

We investigated the mNEPv~\eqref{eq:nepv}.
A variational characterization for the mNEPv is revealed. 
Based on the variational characterization, we provided 
a geometric interpretation of the SCF iterations for solving the mNEPv. 
The geometry of the SCF illustrates the global monotonic convergence of 
the algorithm and leads to a rigorous 
proof of the global convergence of the SCF.
In addition, we presented an inverse-iteration based 
scheme to accelerate the convergence of the SCF.  
Numerical examples demonstrated the 
effectiveness of the accelerated SCF for solving the mNEPv 
arising from different applications. 
By the intrinsic connection between the mNEPv~\eqref{eq:nepv} 
and the aMax~\eqref{eq:maxf}, we developed an NEPv approach 
for solving the aMax. Algorithmically, it allows the use
of state-of-the-art eigensolvers for fast solution of the aMax.

Most results presented in this work can be extended to
the case NEPv~\eqref{eq:nepv} with $h_i$ being non-decreasing and
locally Lipschitz continuous, i.e.,
\begin{equation}\label{eq:loclip}
c_i(t) := \limsup_{t+\delta t\in[\ell_i,u_i] \atop \delta t\to 0}
        \frac{h_i(t+\delta t) -h_i(t)}{\delta t} < \infty
\end{equation}
for all $t\in  [\ell_i,u_i]$ with $\ell_i=\lambda_{\min}(A_i)$,
$u_i=\lambda_{\max}(A_i)$, and $i=1,\dots,m$.
For example, in complete analogy to~\Cref{thm:vc}, it is possible 
to establish 
a variational characterization of such NEPv.
%
%
In addition, we expect this work to serve as the basis for the study of a more 
general class of NEPv in the form of 
\begin{equation}\label{eq:affinenepv}
	H(x)x = \lambda x
	\quad\text{with}\quad
	H(x) = \sum_{i=1}^m h_i(\rqs(x)) A_i,
\end{equation}
with an $m$-tuple of Hermitian matrices $\mathcal A =(A_1,\dots, A_m)$,
$\rqs$ by~\eqref{eq:rho2}, and $h_i:\R^m\to \R$ to be given functions.
Similar to the mNEPv~\eqref{eq:nepv}, 
one can expect to establish variational characterizations
of~\eqref{eq:affinenepv},
at least for particular functions of $h_i$. 
How to extend the theoretical analysis and geometric interpretation of
the SCF to a general NEPv~\eqref{eq:affinenepv}
is left to future research.

%
%
%
%
%
%
%
%
%
%

\bibliographystyle{plain}
\bibliography{nepvrefs}

\end{document}